\documentclass{article}

\usepackage{amsmath,amssymb,amsthm,amsfonts}
\usepackage{mathtools}

\usepackage[a4paper,left=2.5cm,right=2.5cm,top=2.5cm,bottom=2.5cm]{geometry}
\usepackage{titlesec}
\usepackage{fancyhdr}
\usepackage{setspace}
\usepackage{indentfirst}
\usepackage{setspace}
\allowdisplaybreaks

\usepackage{graphicx}
\usepackage{float}
\usepackage{subcaption}
\usepackage{booktabs}
\usepackage{array}
\usepackage{algorithm}
\usepackage{algpseudocode}
\usepackage{tabularx}
\usepackage[inline]{enumitem}
\usepackage{enumitem}
\usepackage{cite}
\usepackage{hyperref}
\usepackage{cleveref}

\usepackage{color}
\usepackage{tocbibind}
\titleformat{\section}{\raggedright\Large\bfseries}{\thesection}{1em}{}
\usepackage{xcolor} 
\usepackage{hyperref} 
\hypersetup{
    colorlinks=true,
    linkcolor=green,
    filecolor=magenta,      
    urlcolor=cyan,
    citecolor=blue,
    pdftitle={Your Document Title},
    pdfpagemode=FullScreen,
    }

\usepackage{authblk}

\usepackage{amsthm}
\theoremstyle{definition}

\newtheorem{theorem}{Theorem}[section] 

\newtheorem{lemma}{Lemma}[section]      

\newtheorem{proposition}{Proposition}[section] 
 
\theoremstyle{definition}
\newtheorem{definition}{Definition}[section] 
  
\theoremstyle{definition}
\newtheorem{remark}{Remark}[section]

\newcommand{\Hilbert}{\mathcal{H}}  
\newcommand{\R}{\mathbb{R}}

\title{Accelerated Convergence of a Second-Order Dynamical System and its Application to Splitting Algorithms for Comonotone Inclusions}

\author[1]{Yan Tang\thanks{The corresponding author. E-mail: tangyan@ctbu.edu.cn}}
\author[1]{Jun Dong}

\affil[1]{School of Mathematics and Statistics, Chongqing Technology and Business University, Chongqing 400067, People's Republic of China}
\date{}
\begin{document}
\maketitle

\begin{abstract}
This paper introduces a novel second-order dynamical system driven by a forward-backward splitting operator for solving the structured inclusion $0\in(A+B)(x)$ in a real Hilbert space, where $A$ is a maximal $\rho$-comonotone operator and $B$ is a $\nu$-cocoercive operator. The well-posedness of the system is established, and Lyapunov analysis yields accelerated convergence rates of order $o(\frac{1}{t})$ for the velocity and $o(\frac{1}{t^2})$ for the forward-backward residual, together with weak convergence of the trajectories to $\operatorname{zer}(A+B)$. Temporal discretization further leads to a class of double inertial Halpern forward-backward splitting algorithms that encompasses the classical forward-backward splitting method and its inertial variants as special cases. Numerical experiments on split feasibility, sparse signal recovery, and image deblurring illustrate the effectiveness of the proposed algorithm.
\end{abstract}



\section{Introduction}
Let $\mathcal{H}$ be a real Hilbert space with inner product $\langle\cdot,\cdot\rangle$ and induced norm $\|\cdot\|$.
We consider the structured inclusion problem of finding $x\in\mathcal{H}$ such that
\begin{equation}\label{eq:intro_inclusion}
0\in(A+B)(x),
\end{equation}
where $A:\mathcal{H}\rightrightarrows\mathcal{H}$ is a set-valued operator, $B:\mathcal{H}\to\mathcal{H}$ is a single-valued operator and the solution set $\operatorname{zer}(A+B) = \{x \in \mathcal{H} : 0 \in (A+B)(x)\}$ is nonempty. This formulation constitutes a central problem of foundational significance, with wide-ranging applications in sparse signal recovery, image processing, and machine learning \cite{Combettes, Duchi, Raguet}.

In recent years, the continuous-time dynamics optimization approach has attracted considerable interest in academic fields and has been applied to several important special cases of the structured inclusion~\eqref{eq:intro_inclusion}~\cite{Bruck1975, attouch2012second, attouch2014dynamical}. As is well known, when $A=\partial f$ is the subdifferential of a proper, convex, and lower semicontinuous function $f$ and $B=\nabla g$ is the gradient of a Fr\'{e}chet differentiable function $g$, not necessarily convex, whose gradient is Lipschitz continuous, the structured inclusion~\eqref{eq:intro_inclusion} reduces to the monotone inclusion $0\in(\partial f+\nabla g)(x)$, which is equivalent to the structured convex minimization problem $\min_{x\in\mathcal{H}}\{f(x)+g(x)\}$. To solve this problem from a dynamic perspective, Bruck~\cite{Bruck1975} introduced the first-order differential inclusion for almost everywhere $t>0$:
\begin{equation}\label{eq:Bruck}
\dot{x}(t)+\partial f(x(t))+\nabla g(x(t))\ni0,
\end{equation}
and proved that each trajectory converges weakly to an element of $\operatorname{argmin}(f+g)$ when $\operatorname{argmin}(f+g)$ is nonempty. The classical forward-backward algorithm can be viewed as the discrete counterpart of system~\eqref{eq:Bruck}, obtained via an implicit discretization of $\partial f$ and an explicit discretization of $\nabla g$.
Furthermore, Attouch, Maing\'{e}, and Redont~\cite{attouch2012second} studied a first-order differential system with two potentials for a.e.\ $t>0$:
\begin{equation}\label{eq:AMR_first}
\left\{\begin{array}{l}
\dot{x}(t)+\partial f(x(t))+a x(t)+b y(t)\ni0,\\[2pt]
\dot{y}(t)-\nabla g(x(t))+a x(t)+b y(t)=0,
\end{array}\right.
\end{equation}
where $y$ is an auxiliary variable and $a,b>0$. By Moreau-Yosida regularization, they proved that system~\eqref{eq:AMR_first} admits a unique global solution, which is the uniform limit on compact intervals of the trajectories generated by the Moreau-Yosida regularization, and that the velocities of the regularized approximations converge weakly in $L^{2}$ to the velocity of the true solution. Relying on this dynamical framework, Attouch, Peypouquet, and Redont~\cite{attouch2014dynamical} derived, via time discretization, a new class of inertial forward-backward splitting algorithms that relax the classical step-size restriction and encompass the gradient-projection method as a special case.

Another application of the dynamical optimization method concerns the structured inclusion problem~\eqref{eq:intro_inclusion} in the case where $A=\nabla f$ is the gradient of a convex differentiable function and $B$ is a cocoercive operator. Specifically, Attouch and Maing\'{e}~\cite{attouch2011asymptotic} introduced the second-order system:
\begin{equation*}\label{eq:AM2011}
\ddot{x}(t)+u\dot{x}(t)+\nabla f(x(t))+B(x(t))=0,
\end{equation*}
where $\nabla f$ is the gradient of a convex, continuously differentiable function and $B$ is a $\lambda$-cocoercive operator. Under the condition $\lambda u^{2}>1$, which involves only the nonpotential operator and the damping parameter, they proved that each trajectory converges weakly to an element of $(\nabla f+B)^{-1}(0)$ and the velocity asymptotically vanishes as $t\to+\infty$.
Drawing on the capacity of Hessian-driven damping to suppress oscillations, Adly, Attouch, and Vo~\cite{adly2021asymptotic} studied a Newton-like inertial dynamics with explicit geometric damping terms:
\begin{equation}\label{eq:AAV2021}
\ddot{x}(t)+u\dot{x}(t)+\nabla f(x(t))+B(x(t))+\beta_{f}\nabla^{2}f(x(t))\dot{x}(t)+\beta_{b}B'(x(t))\dot{x}(t)=0,
\end{equation}
where $B$ is a $\lambda$-cocoercive operator and $u>0$, $\beta_{f}>0$, $\beta_{b}>0$ are damping parameters. The terms $\nabla^{2}f(x(t))\dot{x}(t)$ and $B'(x(t))\dot{x}(t)$ are interpreted respectively as $\frac{d}{dt}(\nabla f(x(t)))$ and $\frac{d}{dt}(B(x(t)))$ in the sense of distributions. Under suitable conditions, they established the well-posedness of this system and proved the weak convergence of the generated trajectories to $\operatorname{zer}(\nabla f+B)$. Moreover, by rewriting this system as an equivalent first-order formulation in time and space, the convergence analysis was further extended to nonsmooth convex potentials.
To naturally yield inertial proximal-gradient splitting algorithms by temporal discretization, they subsequently developed a variant system~\cite{adly2023convergence}:
\begin{equation*}
\label{eq:AAV2023}
\ddot{x}(t)+u\dot{x}(t)+\nabla f\bigl(x(t)+\beta_{f}\dot{x}(t)\bigr)+B\bigl(x(t)+\beta_{b}\dot{x}(t)\bigr)=0,
\end{equation*}
whose counterpart~\eqref{eq:AAV2021} is recovered by a first-order Taylor expansion of the shifted arguments.

A common feature of all the works above is that the set-valued operator $A$ is required to be monotone.
In several relevant settings, however, this requirement is not met. Non-convex composite minimization problems, in which the regularizer is merely weakly convex, give rise to hypomonotone operators that need not be monotone. Applications such as the method of multipliers for nonlinear programming~\cite{combettes2004} and smooth nonconvex-nonconcave minimax problems~\cite{lee2021} give rise to co-hypomonotone operators that need not be monotone. To accommodate such problems while retaining the dynamical systems approach, one needs a notion that relaxes monotonicity in a controlled and principled manner. The concept of $\rho$-comonotonicity, systematically developed by Bauschke, Moursi, and Wang~\cite{bauschke2021generalized}, serves precisely this purpose. Within a single axiomatic scheme, it unifies cocoercivity ($\rho>0$), classical monotonicity ($\rho=0$), and co-hypomonotonicity ($\rho<0$), thereby encompassing the operator classes that arise in the descriptions above.
A crucial structural property is that averaged operators can be characterized as resolvents of comonotone operators under suitable scaling transformations, thereby extending the classical correspondence between maximally monotone operators and firmly nonexpansive mappings to broader frameworks for non-monotone inclusion problems.

While the works described above have significantly advanced the theory of continuous-time inertial dynamics and their algorithmic counterparts, to the best of our knowledge, no prior study has considered a second-order dynamical system whose driving term is the forward-backward splitting operator associated with a maximal $\rho$-comonotone operator and a cocoercive operator from a dynamical perspective. To fill this gap and obtain a class of double inertial Halpern forward-backward splitting algorithms, we propose the following second-order dynamical system for the structured inclusion~\eqref{eq:intro_inclusion}, in which $A$ is assumed to be maximal $\rho$-comonotone and $B$ is $\nu$-cocoercive with $\nu>0$:
\begin{equation*}
\ddot{x}(t)+\alpha(t)\dot{x}(t)+\beta(t)\frac{d}{dt}T_{\lambda(t),\gamma(t)}(x(t))+b(t)T_{\lambda(t),\gamma(t)}(x(t))=0,
\end{equation*}
where $J_{\gamma A}:=(\operatorname{Id}+\gamma A)^{-1}$ denotes the resolvent of $A$, $T_{\lambda,\gamma}:=\frac{1}{\lambda}\bigl(\operatorname{Id}-J_{\gamma A}\circ(\operatorname{Id}-\gamma B)\bigr)$ is the associated forward-backward splitting operator, and $\gamma(t)>0$, $\lambda(t)>0$ are time-dependent parameters for all $t\geq t_{0}$. The time-dependent coefficients $\alpha(t),\beta(t),b(t)$ are expressed through a set of auxiliary parameters whose algebraic relations are specified in the sequel.

The remainder of the paper is organized as follows. Section~\ref{sec:prelim} reviews some notions and basic results used throughout the paper. Section~\ref{sec:existence} proposes a second-order dynamical system for the inclusion problem~\eqref{eq:intro_inclusion} and establishes the existence and uniqueness of global solutions to its Cauchy problem. Section~\ref{sec:asymptotic} is devoted to the asymptotic analysis of the proposed system, establishing accelerated convergence rates and proving the weak convergence of the trajectories. Section~\ref{sec:algorithm} derives, via temporal discretization of the continuous dynamics, a class of double inertial Halpern forward-backward splitting algorithms that encompasses the classical forward-backward splitting algorithm~\cite{Combettes}, the inertial forward-backward splitting algorithm~\cite{lorenz2015inertial}, and the inertial Halpern forward-backward splitting algorithm~\cite{cholamjiak2018inertial}, and validates the effectiveness of the proposed algorithms through several numerical experiments. Section~\ref{sec:conclusion} provides concluding remarks and outlines future research directions.

\section{Preliminaries}
\label{sec:prelim}
Throughout this paper, let $\mathcal{H}$ be a real Hilbert space and $\mathcal{H} \times \mathcal{H}$ be the product space endowed with the inner product $\langle (x, z), (x', z') \rangle := \langle x, x' \rangle_{\mathcal{H}} + \langle z, z' \rangle_{\mathcal{H}}, \forall (x, z), (x', z') \in \mathcal{H} \times \mathcal{H}$. 
It follows that $\mathcal{H} \times \mathcal{H}$ is also a Hilbert space with the induced norm $\|(x, z)\|_{\mathcal{H} \times \mathcal{H}} = \sqrt{\|x\|_{\mathcal{H}}^2 + \|z\|_{\mathcal{H}}^2}$ for all $(x, z) \in \mathcal{H} \times \mathcal{H}$.
\vskip 1mm

Let $A:\mathcal{H} \rightrightarrows \mathcal{H}$ be a set-valued operator. The domain of $A$ is defined as $\mathrm{dom} A := \{x \in \mathcal{H} : Ax \neq \emptyset\}$, its graph as $\mathrm{gra} A := \{(x, u) \in \mathcal{H} \times \mathcal{H} : u \in Ax\}$, and its set of zeros as $\operatorname{zer} A := \{x \in \mathcal{H} : 0 \in Ax\}$. For any parameter $\gamma > 0$, the resolvent of $A$ with index $\gamma$ is defined as $J_{\gamma A} := (\operatorname{Id} + \gamma A)^{-1}$, where $\operatorname{Id}$ denotes the identity operator on $\mathcal{H}$. The Yosida approximation of $A$ is given by $A_\gamma := \frac{1}{\gamma} (\operatorname{Id} - J_{\gamma A})$. 

\begin{definition}
{\rm (See \cite[Definition 2.4]{bauschke2021generalized})}
\label{def_1}
Let $A: \Hilbert \to 2^\Hilbert$ be a set-valued operator and $\rho \in \R$. Then we have the following operator:
\begin{enumerate}
    \item[(1)] $A$ is $\rho$-monotone if
    $
    \langle x-y, u-v \rangle \geq \rho \|x-y\|^2, \,\,\,\forall (x,u),   (y,v) \in \text{gra } A.
   $
    \item[(2)] $A$ is maximal $\rho$-monotone if $A$ is $\rho$-monotone and there is no $\rho$-monotone operator  $V: \Hilbert \to 2^\Hilbert$  such that $\text{gra } V$ properly contains $\text{gra } A$, that is, for all $(x,u) \in \Hilbert \times  2^\Hilbert$,
    $
    (x,u) \in \text{gra } A\, \Longleftrightarrow\, \langle x-y, u-v \rangle \geq \rho \|x-y\|^2,\,\,\,\forall (y,v) \in \text{gra } A.
    $
    \item[(3)] $A$ is $\rho$-comonotone if
    $
    \langle x-y, u-v \rangle \geq \rho \|u-v\|^2, \,\,\,\forall (x,u),  (y,v) \in \text{gra } A.
    $
    \item[(4)] $A$ is maximal $\rho$-comonotone if $A$ is $\rho$-comonotone and there is no $\rho$-comonotone operator $V: \Hilbert\to  2^\Hilbert$ such that $\text{gra } V$ properly contains $\text{gra } A$, that is, for all $(x,u) \in \Hilbert \times  2^\Hilbert$,
    $
    (x,u) \in \text{gra } A\, \Longleftrightarrow\,  \langle x-y, u-v \rangle \geq \rho \|u-v\|^2,\,\,\,\forall (y,v) \in \text{gra } A.
   $
\end{enumerate}
\end{definition}

\begin{remark} (See \cite[Remark 2.5]{bauschke2021generalized})
\label{rem:comonotone}  We have the properties of the opertaor $A$ given in Definition \ref{def_1}:
\leavevmode
\begin{enumerate} 
\item[(1)] When $\rho = 0$, both $\rho$-monotonicity of $A$ and $\rho$-comonotonicity of $A$ reduce to monotonicity of $A$.
\item[(2)] When $\rho < 0$, $\rho$-monotonicity is known as $|\rho|$-hypomonotonicity, see [\cite{rockafellar2009variational}, Example 12.28] and [\cite{burachik2008}, Definition 6.9.1]. In this case, the $\rho$-comonotonicity is also known as $|\rho|$-co-hypomonotonicity (see [\cite{combettes2004}, Definition 2.2]).
\item[(3)] In passing, we point out that when $\rho > 0$, $\rho$-monotonicity of $A$ reduces to $\rho$-strong monotonicity of $A$, while $\rho$-comonotonicity of $A$ reduces to $\rho$-cocoercivity of $A$.
\end{enumerate}
\end{remark}

As demonstrated in \cite[Example 1]{lee2021} and \cite[Example 3]{pethick2023}, the $\rho$-comonotonicity with $\rho < 0$ is strictly weaker than the monotonicity.

\begin{definition} [\cite{bauschke2017convex}]
Let $\mathcal{H}$ be a real Hilbert space.  Recall the following: 
\begin{enumerate}[label=(\roman*)]
    \item[(1)] An operator $A: \mathcal{H} \rightarrow \mathcal{H}$ is said to be $L$-Lipschitz continuous for some $L>0$ if
    $
    \| A x - A y \| \leq L \| x - y \|,\,\,\,  \forall x, y \in \mathcal{H}.
    $
    \item[\rm (2)] An operator $A: \mathcal{H} \rightarrow \mathcal{H}$ is said to be $\alpha$-averaged for some $\alpha \in \left( 0, 1 \right)$ if there exists a nonexpansive operator $R: \mathcal{H} \to \mathcal{H}$ such that
    $
    A = (1 - \alpha)\operatorname{Id} + \alpha R.
    $
    \item[(3)] An operator $A: \mathcal{H} \rightarrow \mathcal{H}$ is said to be nonexpansive if
    $
    \| A x - A y \|_{\mathcal{H}} \leq \| x - y \|_{\mathcal{H}},\,\,\,  \forall x, y \in \mathcal{H}.
    $
    \item[(4)] An operator $A: \mathcal{H} \rightarrow \mathcal{H}$ is said to be $\nu$-cocoercive for some $\nu>0$ if
    $
    \langle A x - A y, x - y \rangle \geq \nu \| A x - A y \|^2,\,\,\,  \forall x, y \in \mathcal{H}.
    $
\end{enumerate}
\end{definition}

If an operator $A: \mathcal{H} \rightarrow \mathcal{H}$ is $\nu$-cocoercive for some $\nu > 0$, then $A$ is $\frac{1}{\nu}$-Lipschitz continuous and monotone. Additionally, if $A$ is $\alpha$-averaged for some $\alpha \in \left( 0, 1 \right)$, then $A$ is nonexpansive.

\begin{proposition} [\cite{bauschke2017convex}] \label{prop:cocoercive_averaged}
Let $D$ be a nonempty subset of $\mathcal{H}$ and $B: D \to \mathcal{H}$ be a single-valued operator. For any $\nu \in \mathbb{R}_{++}$ and parameter $\gamma \in \left( 0, 2\nu \right)$, the operator $B$ is $\nu$-cocoercive if and only if $\operatorname{Id} - \gamma B$ is $\frac{\gamma}{2\nu}$-averaged.
\end{proposition}

\begin{proposition} [\cite{bauschke2017convex}] \label{prop:composition_averaged}
Let $D$ be a nonempty subset of $\mathcal{H}$ and $T_{1}, T_{2} : D \to D$ be two operators. If $T_{1}$ is $\alpha_{1}$-averaged and $T_{2}$ is $\alpha_{2}$-averaged for some $\alpha_{1}, \alpha_{2} \in \left( 0, 1 \right)$, then the composition $T = T_{1} \circ T_{2}$ is $\alpha$-averaged, where $\alpha = \frac{\alpha_{1} + \alpha_{2} - 2\alpha_{1} \alpha_{2}}{1 - \alpha_{1} \alpha_{2}}\in \left( 0, 1 \right)$.
\end{proposition}

\begin{proposition} [\cite{bauschke2017convex}] \label{prop:averaged_characterization}
Let $D$ be a nonempty subset of $\mathcal{H}$, $T: D \to \mathcal{H}$ be a nonexpansive operator, and $\alpha \in \left( 0, 1 \right)$. The following assertions are equivalent:
\begin{enumerate}
    \item[\rm (1)] $T$ is $\alpha$-averaged;
    \item[\rm (2)] $\left( 1 - 1/\alpha \right)\operatorname{Id} + \left( 1/\alpha \right)T$ is nonexpansive;
    \item[\rm (3)] For all $x, y \in D$, it holds that:
    $
    \| Tx - Ty \|_{\mathcal{H}}^2 \leq \| x - y \|_{\mathcal{H}}^2 - \frac{1 - \alpha}{\alpha} \| (\operatorname{Id} - T)x - (\operatorname{Id} - T)y \|_{\mathcal{H}}^2;
    $
\end{enumerate}
In particular, $T$ is $\alpha$-averaged if and only if $\operatorname{Id} - T$ is $\frac{1}{2\alpha}$-cocoercive.
\end{proposition}

\begin{definition} {(See \cite[Definition 2.1]{AttouchSvaiter2011})}
\label{def:absolute_continuity}
Given $b \in \mathbb{R}_{+}$, a function $x : [0, b] \to \mathcal{H}$ is said to be {\it absolutely continuous} if one of the following equivalent properties holds:
\begin{enumerate}
    \item[\rm (1)] There exists an integrable function $y : [0, b] \to \mathcal{H}$ such that $x(t) = x(0) + \int_0^t y(s) ds$ for all $t \in [0, b]$;
    \item[\rm (2)] $x$ is continuous and its distributional derivative belongs to the Lebesgue space $L^1([0, b]; \mathcal{H})$;
    \item[\rm (3)] For every $\epsilon > 0$, there exists some $\eta > 0$ such that, for any finite family of intervals $I_k = (a_k, b_k)$,
    \begin{equation*}
    I_k \cap I_j = \emptyset, \,\, \forall i, j \geq 1, \, i \neq j, \quad \sum |b_k - a_k| \leq \eta \implies \sum \|x(b_k) - x(a_k)\|_{\mathcal{H}} \leq \epsilon.
    \end{equation*}
\end{enumerate}
Moreover, an absolutely continuous function is differentiable almost everywhere.
\end{definition}

\begin{proposition} [\cite{Bartz S2022}] \label{prop:resolvent_single_valued}
Let $A: \mathcal{H} \rightrightarrows \mathcal{H}$ be $\rho$-comonotone and let $\gamma > 0$ be such that $\gamma + \rho > 0$. Then the following virtues hold:
\begin{enumerate}
    \item[\rm (1)] $J_{\gamma A}$ is single-valued;
    \item[\rm (2)] $\mathrm{dom} J_{\gamma A} = \mathcal{H}$ if and only if $A$ is maximal $\rho$-comonotone.
\end{enumerate}
\end{proposition}

\begin{proposition} [\cite{tan2023second}] \label{prop:max_comonotone_properties}
Let $A: \mathcal{H} \rightrightarrows \mathcal{H}$ be $\rho$-comonotone and let $\gamma > \max\{0, -2\rho\}$. For any $x, y \in \mathcal{H}$, the following properties hold:
\begin{enumerate}
    \item[\rm (1)] $A$ is maximal $\rho$-comonotone if and only if $A_\gamma$ is $\left( \rho + \gamma \right)$-cocoercive;
    \item[\rm (2)] $J_{\gamma A} : \mathcal{H} \to \mathcal{H}$ is $\frac{\gamma}{2(\rho + \gamma)}$-averaged and $A_\gamma : \mathcal{H} \to \mathcal{H}$ is $\frac{1}{\rho + \gamma}$-Lipschitz continuous;
\end{enumerate}
\end{proposition}

\begin{lemma}
\label{thm:coco_T_final}
Let $A: \mathcal{H} \to 2^{\mathcal{H}}$ be a $\rho$-comonotone operator and $B: \mathcal{H} \to \mathcal{H}$ be a $\nu>0$-cocoercive operator. For any parameters $\lambda > 0$ and $\gamma \in \left( \max\{0, -2\rho\}, \, 2\nu \right)$, the splitting operator defined by $T_{\lambda, \gamma} := \frac{1}{\lambda} \left[ \operatorname{Id} - J_{\gamma A} \circ \left( \operatorname{Id} - \gamma B \right) \right]$ is $\Lambda$-cocoercive on $\mathcal{H}$, where $\Lambda := \lambda \cdot \frac{4\nu \left( \rho + \gamma \right) - \gamma^2}{4\gamma \left( \nu + \rho \right)}>0$.
\end{lemma}

\begin{proof}
According to Proposition \ref{prop:max_comonotone_properties} (2), since the operator $A$ is $\rho$-comonotone and $\gamma > -2\rho$, its resolvent $J_{\gamma A} = (\operatorname{Id} + \gamma A)^{-1}$ is $\alpha_1$-averaged with the coefficient $\alpha_1 = \frac{\gamma}{2\left( \rho + \gamma \right)} \in(0,1)$.
Concurrently, since the operator $B$ is $\nu$-cocoercive and the parameter satisfies $\gamma \in (0, 2\nu)$, Proposition \ref{prop:cocoercive_averaged} guarantees that the mapping $\operatorname{Id} - \gamma B$ is $\alpha_2$-averaged with the coefficient $\alpha_2 = \frac{\gamma}{2\nu}\in(0,1)$. 
Furthermore, applying Proposition \ref{prop:composition_averaged} to the composition mapping $J_{\gamma A} \circ \left( \operatorname{Id} - \gamma B \right)$ leads to its averageness with $\alpha'= \frac{\alpha_{1} + \alpha_{2} - 2\alpha_{1} \alpha_{2}}{1 - \alpha_{1} \alpha_{2}}=\frac{2\gamma\left( \nu + \rho \right)}{4\nu\left( \rho + \gamma \right) - \gamma^2}\in \left( 0, 1 \right)$. Thus, it follows from Proposition \ref{prop:averaged_characterization} and the definition of $T_{\lambda, \gamma}$ that the splitting operator $T_{\lambda, \gamma}$ is $\Lambda :=\frac{1}{2\alpha^{\prime}}= \lambda \cdot \frac{4\nu \left( \rho + \gamma \right) - \gamma^2}{4\gamma \left( \nu + \rho \right)}>0$-cocoercive on $\mathcal{H}$.
\end{proof}

\begin{lemma}(See \cite[Proposition 6.2.1]{Haraux A1991})
\label{cauchy-lipschitz}
Let $\mathbb{X}$ be a real Banach space and $F: [t_{0}, +\infty) \times \mathbb{X} \to \mathbb{X}$ be a function. Suppose that
\begin{itemize}
\item[{\rm(i)}] $F(t, \cdot): \mathbb{X} \to \mathbb{X}$ is continuous and $\|F(t,x)-F(t,y)\|\leq M(t,\|x\|+\|y\|)\|x-y\|$ for all $x,y\in\mathbb{X}$, where $M(t,l)\in L^{1}_{\mathrm{loc}}([t_{0}, +\infty))$ for all $l>0$, for almost all $t \in [t_{0}, +\infty)$, where $L_{\mathrm{loc}}^1([t_0,+\infty))$ denotes the family of locally integrable functions on $[t_0,+\infty)$;
    \item[{\rm(ii)}] for every $x \in \mathbb{X}$, $F(\cdot, x) \in L^{1}_{\mathrm{loc}}([t_{0}, +\infty))$;
    \item[{\rm(iii)}] $F(t, \cdot): \mathbb{X} \to \mathbb{X}$ satisfies
$\|F(t, x)\| \leq P(t)(1 + \|x\|),\,\,\,\,  P(t) \in L^{1}_{\mathrm{loc}}([t_{0}, +\infty))$ for almost all $t \in [t_{0}, +\infty)$.
\end{itemize}
Then,  the equation $\frac{d}{dt}x(t) = F(t, x(t))$, $x(t_{0}) = x_{0}$ has a unique global trajectory $x: [t_{0}, +\infty) \to \mathbb{X}$.
\end{lemma}

\begin{lemma} (See \cite[Lemma A.2]{Attouch H_2019})
\label{lem:opial}
Let $S$ be a nonempty subset of $\mathcal{H}$ and let $x : [t_0, +\infty) \to \mathcal{H}$. Assume that
\begin{enumerate}
\item[{\rm(i)}] for every $x^* \in S$, $\lim_{t \to \infty} \|x(t) - x^*\|$ exists;
\item[{\rm(ii)}] every weak sequential limit point of $x(t)$ as $t \to +\infty$, belongs to $S$.
\end{enumerate}
Then $x(t)$ converges weakly as $t \to +\infty$ to a point in $S$.
\end{lemma}

\begin{lemma}
\label{lem:weighted_decay_criterion}
Let $r:[t_0,+\infty)\to(0,+\infty)$ be locally integrable and satisfy
$\int_{t_0}^{+\infty}r(t)dt=+\infty$. Let $F:[t_0,+\infty)\to\mathcal{H}$ be locally absolutely continuous. Assume that $\int_{t_0}^{+\infty} r(t)\|F(t)\|_{\mathcal{H}}^2dt<+\infty$ and that, for some $t_1\ge t_0$ and $L>0$, $\|\dot F(t)\|_{\mathcal{H}}\le Lr(t)$ 
for almost every $t\ge t_1$. Then
$\lim_{t\to+\infty}\|F(t)\|_{\mathcal{H}}=0$.
\end{lemma}

\begin{proof}
It is enough to work on the tail $[t_1,+\infty)$. Define
\[
\tau:[t_1,+\infty)\to[0,+\infty),\qquad
\tau(t):=\int_{t_1}^{t}r(u)du .
\]
Since $r(t)>0$ for almost every $t > s \geq t_1$, one has
\[
\tau(t)-\tau(s)=\int_s^t r(u)du>0.
\]
Furthermore, $\tau(t)$ is continuous and $\tau(t)\to+\infty$ as $t\to+\infty$. Hence $\tau(t)$ is a continuous strictly increasing bijection from $[t_1,+\infty)$ onto $[0,+\infty)$, and its inverse $\tau^{-1}(t)$ is well defined. Let $G(s):=F(\tau^{-1}(s))$. For $0\le s_1<s_2$, with $t_i=\tau^{-1}(s_i)$, the derivative bound gives
\[
\|G(s_2)-G(s_1)\|_{\mathcal{H}}
\le \int_{t_1}^{t_2}\|\dot F(t)\|_{\mathcal{H}}dt
\le L\int_{t_1}^{t_2}r(t)dt
=L(s_2-s_1).
\]
Thus $G(s)$ is $L$-Lipschitz on $[0,+\infty)$. Moreover, the change of variables $s=\tau(t)$ gives
\[
\int_0^{+\infty}\|G(s)\|_{\mathcal{H}}^2ds
=\int_{t_1}^{+\infty}r(t)\|F(t)\|_{\mathcal{H}}^2dt<+\infty.
\]
Assume that $G(s)$ did not converge strongly to $0$, then there would exist $\varepsilon>0$ and $s_n\to+\infty$ such that $\|G(s_n)\|_{\mathcal{H}}\ge\varepsilon$. Passing to a subsequence, we may suppose that $s_n\ge \varepsilon/(4L)$ and that the intervals
$I_n=[s_n-\varepsilon/(4L),s_n+\varepsilon/(4L)]$ are pairwise disjoint. For every $s\in I_n$, one has $|s-s_n|\le \varepsilon/(4L)$, and the $L$-Lipschitz continuity of $G$ gives
\[
\|G(s)-G(s_n)\|_{\mathcal{H}}\le L|s-s_n|\le\frac{\varepsilon}{4}.
\]
Morever, utilizing the reverse triangle inequality yields
\[
\|G(s)\|_{\mathcal{H}}
\ge \|G(s_n)\|_{\mathcal{H}}-\|G(s)-G(s_n)\|_{\mathcal{H}}
\ge \varepsilon-\frac{\varepsilon}{4}
=\frac{3\varepsilon}{4}\ge\frac{\varepsilon}{2}.
\]
Hence
\[
\int_0^{+\infty}\|G(s)\|_{\mathcal{H}}^2ds
\ge \sum_n \int_{I_n}\frac{\varepsilon^2}{4}ds=+\infty,
\]
which contradicts the square integrability of $G(s)$. Since $F(t)=G(\tau(t))$ and $\tau(t)\to+\infty$ as $t\to+\infty$, it follows that $F(t)\to0$ as $t\to+\infty$.
\end{proof}

\section{Well-posedness of System\texorpdfstring{~\eqref{eq:SHDS}}{ (SHDS)}}
\label{sec:existence}

In this section, we investigate the existence and uniqueness of trajectories generated by the proposed dynamical system:
\begin{equation}
\label{eq:SHDS}
\ddot{x}(t) + \alpha(t)\dot{x}(t) + \beta(t)\frac{d}{dt}T_{\lambda(t), \gamma(t)}(x(t)) + b(t)T_{\lambda(t), \gamma(t)}(x(t)) = 0,
\end{equation}
where $T_{\lambda(t), \gamma(t)} = \frac{1}{\lambda(t)}(\operatorname{Id} - J_{\gamma(t) A} \circ (\operatorname{Id} - \gamma(t) B))$ is the $\Lambda(t) := \lambda(t) \cdot \frac{4\nu \left( \rho + \gamma(t) \right) - \gamma^2(t)}{4\gamma(t) \left( \nu + \rho \right)}$-cocoercive operator according to Lemma \ref{thm:coco_T_final}. The time-dependent coefficient functions $\alpha(t), \beta(t),$ and $b(t)$ are given by the following forms:
\begin{equation}\label{eq:coefficients_definition}
\left\{
\begin{aligned}
\alpha(t) &= \frac{\delta(t) - \dot{\Gamma}(t)}{1 - \Gamma(t)} - \frac{\dot{\delta}(t)}{\delta(t)}, \\
\beta(t) &= \frac{\lambda(t) \theta(t)}{1 - \Gamma(t)}, \\
b(t) &= \frac{\frac{d}{dt}\big(\lambda(t)\theta(t)\big) - \lambda(t)\theta(t) \frac{\dot{\delta}(t)}{\delta(t)}}{1 - \Gamma(t)}.
\end{aligned}
\right.
\end{equation}
where $\Gamma(t) := \theta(t)\phi(t)+(1-\theta(t))\psi(t)<1$.

The following assumptions are fundamental to our theoretical framework and are assumed to hold throughout this paper, unless explicitly stated otherwise:
\vskip 2mm
{\bf (H1)} The operator $A: \mathcal{H} \to 2^{\mathcal{H}}$ is maximally $\rho$-comonotone and $B: \mathcal{H} \to \mathcal{H}$ is $\nu$-cocoercive with $\nu > 0$. The parameter functions $\lambda(\cdot), \gamma(\cdot), \delta(\cdot), \theta(\cdot), \phi(\cdot), \psi(\cdot): [t_0, +\infty) \to \mathbb{R}$ are continuously differentiable and satisfy $\lambda(t) > 0$, $\delta(t) > 0$, $\theta(t) \in [0,1]$, $\phi(t) \ge 0$, $\psi(t) \ge 0$, $\gamma(t) \in [\underline{\gamma}, \bar{\gamma}] \subset \left( \max\{0, -2\rho\}, \, 2\nu \right)$, and $\inf_{t \geq t_0} \Lambda(t) > 0$;
\vskip 1mm

{\bf (H2)} The set of zeros $\operatorname{zer}(A+B) \neq \emptyset$, and $\bar{x}$ denotes an element in $\operatorname{zer}(A+B)$;
\vskip 1mm

{\bf (H3)} The coefficients satisfy $\alpha(t) > 0$, $b(t) \ge 0$, $\beta(t) > 0$; $\alpha(t), b(t)$ are locally integrable and $\beta(t)$ is locally absolutely continuous for all $t \ge t_0$.
\vskip 1mm

Before establishing the existence and uniqueness of a strong global solution to the system \eqref{eq:SHDS}, the definition of a strong global solution is introduced.
\begin{definition}
\label{def:strong_global_solution}
We say that \( x: [t_0, +\infty) \to \mathcal{H} \) is a {\it strong global solution} of the system \eqref{eq:SHDS} with the Cauchy data \( (x_0, u_0) \in \mathcal{H} \times \mathcal{H} \) if
\begin{enumerate}[label=(\roman*)]
\item $x(\cdot), \dot{x}(\cdot):[t_0, +\infty) \to \mathcal{H}$ are continuous and absolutely continuous on each bounded interval $[t_0, T]$ for $t_0 < T < +\infty$;
\item $\ddot{x}(t) + \alpha(t)\dot{x}(t) + \beta(t)\frac{d}{dt}T_{\lambda(t), \gamma(t)}(x(t)) + b(t)T_{\lambda(t), \gamma(t)}(x(t)) = 0$ for almost every $t \in [t_0, +\infty)$;
\item $x(t_0) = x_0$ and $\dot{x}(t_0) = u_0$.
\end{enumerate}
\end{definition}

\begin{theorem}
\label{thm:First-order formulation} 
Under Assumptions {\rm {\bf (H1)--(H3)}}, for any initial condition $(x_0, u_0) \in \mathcal{H}\times\mathcal{H}$, the following assertions are equivalent:
\vskip 1mm

{\rm(1)} A trajectory $ x(\cdot):[t_0,+\infty)\to\mathcal{H} $ satisfies the second order Cauchy problem: for almost every $t \geq t_0$:
\begin{equation}
\label{eq:SHDS_second_order}
\left\{
\begin{aligned}
&\ddot{x}(t) + \alpha(t)\dot{x}(t) + \beta(t)\frac{d}{dt}T_{\lambda(t), \gamma(t)}(x(t)) + b(t)T_{\lambda(t), \gamma(t)}(x(t)) = 0, \\
&x(t_0) = x_0, \,\,\, \dot{x}(t_0) = u_0.
\end{aligned}
\right.
\end{equation}
    
{\rm(2)} A pair of trajectories $ (x(\cdot),z(\cdot)):[t_0,+\infty)\to\mathcal{H}\times\mathcal{H} $ satisfies the equivalent first order formulation:
\begin{equation}
\label{eq:SHDS_first_order}
\left\{
\begin{aligned}
&\dot{x}(t) + \beta(t)T_{\lambda(t), \gamma(t)}(x(t)) + z(t) = 0, \\
&\dot{z}(t) - \alpha(t)\dot{x}(t) - \big(b(t) - \dot{\beta}(t)\big)T_{\lambda(t), \gamma(t)}(x(t)) = 0, \\
&x(t_0) = x_0, \,\,\, z(t_0) = -u_0 - \beta(t_0)T_{\lambda(t_0), \gamma(t_0)}(x_0).
\end{aligned}
\right.
\end{equation}
\end{theorem}

\begin{proof}
In fact, the system \eqref{eq:SHDS} can be rewritten as:
\begin{equation}\label{eq:proof_rearranged}
\begin{aligned}
\ddot{x}(t) + \alpha(t)\dot{x}(t) + \frac{d}{dt}\left(\beta(t)T_{\lambda(t), \gamma(t)}(x(t))\right) + \left(b(t) - \dot{\beta}(t)\right)T_{\lambda(t), \gamma(t)}(x(t)) = 0,
\end{aligned}
\end{equation}
for almost every $t \geq t_0$. Integrating \eqref{eq:proof_rearranged} from $t_0$ to $t \geq t_0$, with the initial condition $\dot{x}(t_0) = u_0$, we obtain:
\begin{equation*}
\label{eq:proof_integrated}
\begin{aligned}
\dot{x}(t) + \beta(t)T_{\lambda(t), \gamma(t)}(x(t)) + z(t) = 0,
\end{aligned}
\end{equation*}
where
\begin{equation*}
\begin{aligned}
z(t) := \int_{t_0}^{t} \alpha(s)\dot{x}(s)ds + \int_{t_0}^{t} \left(b(s) - \dot{\beta}(s)\right)T_{\lambda(s), \gamma(s)}(x(s))ds - u_0 - \beta(t_0)T_{\lambda(t_0), \gamma(t_0)}(x_0),
\end{aligned}
\end{equation*}
which admits the derivative $\dot{z}(t) = \alpha(t)\dot{x}(t) + \big(b(t) - \dot{\beta}(t)\big)T_{\lambda(t), \gamma(t)}(x(t))$.
\vskip 1mm
If $x(t)$ is a solution of the system \eqref{eq:SHDS} with the initial conditions $x(t_0) = x_0$ and $\dot{x}(t_0) = u_0$, then the definition of $z(t)$ gives that the pair $(x(t), z(t))$ satisfies the first order system \eqref{eq:SHDS_first_order} with $x(t_0) = x_0$ and $z(t_0) = - u_0 - \beta(t_0)T_{\lambda(t_0), \gamma(t_0)}(x_0)$.
\vskip 1mm
Conversely, assume  $(x(t), z(t))$ 
satisfies the first-order system \eqref{eq:SHDS_first_order} with the given initial conditions. Differentiating $\dot{x}(t)$
and substituting  $\dot{z}(t)$ recovers system \eqref{eq:SHDS}, with initial conditions consistent by construction. This completes the proof.
\end{proof}

\begin{theorem}
\label{thm:existence-uniqueness_SHDS}
Under Assumptions {\rm {\bf (H1)--(H3)}}, for any initial point $(x_0, u_0) \in \mathcal{H} \times \mathcal{H}$, there exists a unique strong global solution $x(\cdot) : [t_0, +\infty) \to \mathcal{H}$ to the second order Cauchy problem for almost every $t \geq t_0$:
\begin{equation*}
\left\{
\begin{aligned}
&\ddot{x}(t) + \alpha(t)\dot{x}(t) + \beta(t)\frac{d}{dt}T_{\lambda(t), \gamma(t)}(x(t)) + b(t)T_{\lambda(t), \gamma(t)}(x(t)) = 0, \\
&x(t_0) = x_0, \,\,\, \dot{x}(t_0) = u_0.
\end{aligned}
\right.
\end{equation*}
\end{theorem}

\begin{proof}
By Theorem \ref{thm:First-order formulation}, the second order Cauchy problem \eqref{eq:SHDS_second_order} is equivalent to the following first-order system:
\begin{equation}
\label{eq:first-order_system_SHDS}
\begin{cases}
\dot{x}(t) + \beta(t)T_{\lambda(t), \gamma(t)}(x(t)) + z(t) = 0, \\
\dot{z}(t) = \alpha(t)\dot{x}(t) + \big(b(t) - \dot{\beta}(t)\big)T_{\lambda(t), \gamma(t)}(x(t)), \\
x(t_0) = x_0, \,\,\, z(t_0) = -u_0 - \beta(t_0)T_{\lambda(t_0), \gamma(t_0)}(x_0).
\end{cases}
\end{equation}
By substituting $\dot{x}(t) = -z(t) - \beta(t)T_{\lambda(t), \gamma(t)}(x(t))$ into the second equation of \eqref{eq:first-order_system_SHDS}, we define the state variable $X(t) = (x(t), z(t))$ and the associated vector field $F: [t_0, +\infty) \times \mathcal{H} \times \mathcal{H} \to \mathcal{H} \times \mathcal{H}$ as follows:
\begin{equation}
\label{eq:vector_field_F_SHDS}
F(t, x, z) := \left( -z - \beta(t)T_{\lambda(t), \gamma(t)}(x), \,\, -\alpha(t)z + \big(b(t) - \dot{\beta}(t) - \alpha(t)\beta(t)\big)T_{\lambda(t), \gamma(t)}(x) \right).
\end{equation}
Then, the system \eqref{eq:first-order_system_SHDS} can be reformulated as the Cauchy problem:
\begin{equation*}
\begin{cases}
\dot{X}(t) = F(t, X(t)),\\
X(t_0) = X_0 := \left( x_0, \,-u_0 - \beta(t_0)T_{\lambda(t_0), \gamma(t_0)}(x_0) \right).
\end{cases}
\end{equation*}

Now, we verify that the vector field $F$ satisfies the three conditions of Lemma \ref{cauchy-lipschitz}.

\noindent \textbf{Local Lipschitz Continuity.} 
For any $(x, z), (x', z') \in \mathcal{H} \times \mathcal{H}$ and fixed $t \geq t_0$, by the definition of the vector field $F$ in \eqref{eq:vector_field_F_SHDS} and the inequality $\|u + v\|_{\mathcal{H}}^2 \leq 2\|u\|_{\mathcal{H}}^2 + 2\|v\|_{\mathcal{H}}^2$ for any $u, v \in \mathcal{H}$, it follows that:
\begin{align*}
& \|F(t, x, z) - F(t, x', z')\|_{\mathcal{H} \times \mathcal{H}}^2 \\
& \quad = \left\| -(z - z') - \beta(t) \big(T_{\lambda(t), \gamma(t)}(x) - T_{\lambda(t), \gamma(t)}(x')\big) \right\|_{\mathcal{H}}^2 \\
& \quad \quad + \left\| -\alpha(t)(z - z') + \left( b(t) - \dot{\beta}(t) - \alpha(t)\beta(t) \right) \big(T_{\lambda(t), \gamma(t)}(x) - T_{\lambda(t), \gamma(t)}(x')\big) \right\|_{\mathcal{H}}^2 \\
& \quad \le \left( 2 + 2\alpha^2(t) \right) \|z - z'\|_{\mathcal{H}}^2 \\
& \quad \quad + 2 \left( \beta^2(t) + \left( b(t) - \dot{\beta}(t) - \alpha(t)\beta(t) \right)^2 \right) \left\|T_{\lambda(t), \gamma(t)}(x) - T_{\lambda(t), \gamma(t)}(x')\right\|_{\mathcal{H}}^2.
\end{align*}
Setting $\underline{\Lambda} := \inf_{t \geq t_0} \Lambda(t) > 0$. By the $\frac{1}{\Lambda(t)}$-Lipschitz continuity of $T_{\lambda(t), \gamma(t)}$ on $\mathcal{H}$, we have
\begin{align*}
\|F(t, x, z) - F(t, x', z')\|_{\mathcal{H} \times \mathcal{H}}^2 &\le C_1(t) \|z - z'\|_{\mathcal{H}}^2 + C_2(t) \frac{1}{\underline{\Lambda}^2} \|x - x'\|_{\mathcal{H}}^2 \\
&\le L_F^2(t) \left( \|x - x'\|_{\mathcal{H}}^2 + \|z - z'\|_{\mathcal{H}}^2 \right) \\
&= L_F^2(t) \|(x, z) - (x', z')\|_{\mathcal{H} \times \mathcal{H}}^2,
\end{align*}
where the time-dependent coefficients are defined as
\begin{equation*}
\left\{
\begin{aligned}
C_1(t) &:= 2 + 2\alpha^2(t), \\
C_2(t) &:= 2\beta^2(t) + 2\left( b(t) - \dot{\beta}(t) - \alpha(t)\beta(t) \right)^2,
\end{aligned}
\right.
\end{equation*}
and the Lipschitz modulus of the vector field is $L_F(t) = \sqrt{\max\left( C_1(t), \frac{1}{\underline{\Lambda}^2} C_2(t) \right)}$. Under Assumption {\bf (H3)}, all component functions of $L_F(t)$ are locally integrable, which ensures that $L_F(t) \in L^1_{\mathrm{loc}}([t_0, +\infty))$.

\noindent \textbf{Linear Growth and Local Integrability.} 
For any state $(x, z) \in \mathcal{H} \times \mathcal{H}$, utilizing the property $T_{\lambda(t), \gamma(t)}(\bar{x}) = 0$ and triangle inequality $\|x - \bar{x}\|_{\mathcal{H}} \le \|x\|_{\mathcal{H}} + \|\bar{x}\|_{\mathcal{H}}$, the vector field \eqref{eq:vector_field_F_SHDS} satisfies:
\begin{align*}
\|F(t, x, z)\|_{\mathcal{H} \times \mathcal{H}} &= \left\| -z - \beta(t)T_{\lambda(t), \gamma(t)}(x) \right\|_{\mathcal{H}} + \left\| -\alpha(t)z + \left( b(t) - \dot{\beta}(t) - \alpha(t)\beta(t) \right) T_{\lambda(t), \gamma(t)}(x) \right\|_{\mathcal{H}} \\
&\le \left( 1 + |\alpha(t)| \right) \|z\|_{\mathcal{H}} + \left( |\beta(t)| + \left| b(t) - \dot{\beta}(t) - \alpha(t)\beta(t) \right| \right) \left\| T_{\lambda(t), \gamma(t)}(x) \right\|_{\mathcal{H}} \\
&\le \left( 1 + |\alpha(t)| \right) \|z\|_{\mathcal{H}} + \left( |\beta(t)| + \left| b(t) - \dot{\beta}(t) - \alpha(t)\beta(t) \right| \right) L(t) \left( \|x\|_{\mathcal{H}} + \|\bar{x}\|_{\mathcal{H}} \right) \\
&\le P(t) \left( 1 + \|x\|_{\mathcal{H}} + \|z\|_{\mathcal{H}} \right) \\
&= P(t) \left( 1 + \|(x, z)\|_{\mathcal{H} \times \mathcal{H}} \right),
\end{align*}
where the coefficient is defined as:
\begin{equation*}
P(t) := \max \left( 1 + |\alpha(t)|, \;\; \left( |\beta(t)| + \left| b(t) - \dot{\beta}(t) - \alpha(t)\beta(t) \right| \right) L(t) \left( 1 + \|\bar{x}\|_{\mathcal{H}} \right) \right).
\end{equation*}
Since $P(t) \in L^1_{\mathrm{loc}}([t_0, +\infty))$ by Assumptions {\bf (H1)} and {\bf (H3)}, the linear growth bound ensures that $F(\cdot, x, z) \in L^1_{\mathrm{loc}}([t_0, +\infty), \mathcal{H} \times \mathcal{H})$ for any fixed state.

\noindent \textbf{Unique strong solution.} Since all three conditions of Lemma \ref{cauchy-lipschitz} are satisfied, there exists a unique global solution to the Cauchy problem associated with the system \eqref{eq:SHDS} according to Theorem \ref{thm:First-order formulation}.
\end{proof}

\section{Asymptotic Approximation of the Trajectories}
\label{sec:asymptotic}
In this section, we investigate the asymptotic properties of the trajectories generated by system \eqref{eq:SHDS}. Before establishing the convergence results, we provide some technical lemmas that quantify the sensitivity of the splitting operator with respect to its parameters.

\begin{lemma}
\label{prop:omega_uniform_positivity}
Under Assumption {\bf (H1)}, the function 
$\Omega(t) := \Omega(\gamma(t)) = \frac{4\nu(\rho + \gamma(t)) - \gamma^2(t)}{4\gamma(t)(\nu + \rho)}$ 
is bounded away from zero on $[t_0, +\infty)$. That is, there exists a constant $\underline{\Omega} > 0$ such that $\Omega(t) \ge \underline{\Omega}$ for all $t \ge t_0$.
\end{lemma}

\begin{proof}
As established in Lemma \ref{thm:coco_T_final},  the mapping $\gamma\mapsto\Omega(\gamma)=\frac{4\nu \left( \rho + \gamma \right) - \gamma^2}{4\gamma \left( \nu + \rho \right)}$ is positive for all $\gamma \in I := \left( \max\{0, -2\rho\}, \, 2\nu \right)$. Furthermore, by Assumption {\rm {\bf (H1)}}, the parameter $\gamma(t)$ is confined within a compact sub-interval $K := [\underline{\gamma}, \bar{\gamma}] \subset I$ for all $t \geq t_0$. According to the Extreme Value Theorem, the continuous function $\Omega(t)$ restricted to the compact set $K$ attains a global minimum at some point $\gamma^* \in K$. Since $\gamma^* \in K \subset I$, the function value $\Omega(\gamma^*)$ is positive. Thus, there exists a uniform constant $\inf_{t \geq t_0} \Omega(t) = \Omega(\gamma^*) =: \underline{\Omega} > 0$ such that $\Omega(t) \ge \underline{\Omega}$ for all $t \ge t_0$. 
\end{proof}

\begin{lemma}
\label{lem:sharp_sensitivity_final}
Under Assumption {\rm {\bf (H1)}}, let $A: \mathcal{H} \to 2^{\mathcal{H}}$ be a maximal $\rho$-comonotone operator and $B: \mathcal{H} \to \mathcal{H}$ be a $\nu$-cocoercive operator with $\nu > 0$. For any $\lambda_1, \lambda_2 > 0$ and $\gamma_1, \gamma_2 \in \left( \max(0, -2\rho), \, 2\nu \right)$, let $\Lambda(t) := \lambda(t) \Omega(t) > 0$ be the cocoercivity constant of $T_{\lambda(t), \gamma(t)}$ with $\Omega(t) := \frac{4\nu \left( \rho + \gamma(t) \right) - \gamma^2(t)}{4\gamma(t) \left( \nu + \rho \right)}$. For any $\bar{x} \in \operatorname{zer}(A+B)$ and $x, y, z \in \mathcal{H}$, the following assertions hold:

{\rm (i)} The resolvent operator satisfies the following inequality:
\begin{equation}\label{eq:resolvent_sensitivity_standard}
\| J_{\gamma_1 A} z - J_{\gamma_2 A} z \|_{\mathcal{H}} \leq \frac{|\gamma_1 - \gamma_2|}{\gamma_1} \| z - J_{\gamma_1 A} z \|_{\mathcal{H}}.
\end{equation}

{\rm (ii)} The difference of the operators satisfies:
\begin{align}
\|\lambda_{1} T_{\lambda_{1}, \gamma_{1}}(x) - \lambda_{2} T_{\lambda_{2}, \gamma_{2}}(y)\|_{\mathcal{H}} \leq \;& 4 \left( \frac{\lambda_1}{\Lambda_1} + 4 \right)\frac{\nu}{\gamma_1} \|x - y\|_{\mathcal{H}} + \frac{2 |\gamma_{1} - \gamma_{2}|}{\gamma_{1}} \| x - \bar{x} \|_{\mathcal{H}} \nonumber \\
& + \frac{8\nu |\gamma_{1} - \gamma_{2}|}{\gamma_{1}} \|B(x)\|_{\mathcal{H}}, \label{eq:asymptotic_consistent_form}
\end{align}
where $\Lambda_1$ denotes $\Lambda(t_1)$ evaluated at parameters $(\lambda_1, \gamma_1)$.

{\rm (iii)} For any strong global solution $x(t)$ of system \eqref{eq:SHDS}, the derivative of the operator satisfies for almost every $t \ge t_0$:
\begin{equation}\label{eq:final_CL_derivative}
\left\| \frac{d}{dt} \left( \lambda(t) T_{\lambda(t), \gamma(t)}(x(t)) \right) \right\|_{\mathcal{H}} \leq 4 \left( \frac{1}{\underline{\Omega}} + 4 \right) \frac{\nu}{\gamma(t)} \|\dot{x}(t)\|_{\mathcal{H}} + \frac{2 |\dot{\gamma}(t)|}{\gamma(t)} \| x(t) - \bar{x} \|_{\mathcal{H}} + \frac{8\nu |\dot{\gamma}(t)|}{\gamma(t)} \| B(x(t)) \|_{\mathcal{H}}.
\end{equation}
\end{lemma}

\begin{proof}
We first recall an algebraic identity that holds for any operator $A$ and parameters $\gamma_1, \gamma_2 > 0$:
\begin{equation}
\label{eq:resolvent_identity_proof}
J_{\gamma_1 A} z = J_{\gamma_2 A} \left( \frac{\gamma_2}{\gamma_1} z + \left( 1 - \frac{\gamma_2}{\gamma_1} \right) J_{\gamma_1 A} z \right), \quad \forall z \in \mathcal{H}.
\end{equation}
By utilizing the algebraic identity \eqref{eq:resolvent_identity_proof} and applying the non-expansiveness of the resolvent $J_{\gamma_2 A}$ ensured by Proposition \ref{prop:max_comonotone_properties} under Assumption {\bf (H1)}, we have
\begin{align}
\left\| J_{\gamma_1 A} z - J_{\gamma_2 A} z \right\|_{\mathcal{H}}
&\le \left\| \left( \frac{\gamma_2}{\gamma_1} z + \left( 1 - \frac{\gamma_2}{\gamma_1} \right) J_{\gamma_1 A} z \right) - z \right\|_{\mathcal{H}} \nonumber \\
&= \frac{|\gamma_1 - \gamma_2|}{\gamma_1} \left\| z - J_{\gamma_1 A} z \right\|_{\mathcal{H}},
\label{eq:resolvent_sensitivity_chain}
\end{align}
which completes assertion (i).

Furthermore, applying triangle inequality to the splitting operators gives
\begin{equation}
\label{eq:proof_T_split_final}
\|\lambda_1 T_{\lambda_1, \gamma_1}(x) - \lambda_2 T_{\lambda_2, \gamma_2}(y)\|_{\mathcal{H}} \leq \|\lambda_1 T_{\lambda_1, \gamma_1}(x) - \lambda_1 T_{\lambda_1, \gamma_1}(y)\|_{\mathcal{H}} + \|\lambda_1 T_{\lambda_1, \gamma_1}(y) - \lambda_2 T_{\lambda_2, \gamma_2}(y)\|_{\mathcal{H}}.
\end{equation}
By Assumption {\bf (H1)}, the splitting operator $T_{\lambda_1, \gamma_1}$ is $\Lambda_1$-cocoercive, which implies $\lambda_1 T_{\lambda_1, \gamma_1}$ is $\left( \frac{\lambda_1}{\Lambda_1} \right)$-Lipschitz continuous. Thus, the first term in \eqref{eq:proof_T_split_final} satisfies:
\begin{equation}\label{eq:proof_T_lip_term}
\|\lambda_1 T_{\lambda_1, \gamma_1}(x) - \lambda_1 T_{\lambda_1, \gamma_1}(y)\|_{\mathcal{H}} \leq \frac{\lambda_1}{\Lambda_1} \|x - y\|_{\mathcal{H}}.
\end{equation}

\noindent For the second term in \eqref{eq:proof_T_split_final}, we denote $z_{1} = y - \gamma_{1} B(y)$ and $z_{2} = y - \gamma_{2} B(y)$. Based on the definition of the splitting operator, the non-expansiveness of $J_{\gamma_1 A}$ and the resolvent identity \eqref{eq:resolvent_sensitivity_standard}, we have:
\begin{align}
\|\lambda_1 T_{\lambda_1, \gamma_1}(y) - \lambda_2 T_{\lambda_2, \gamma_2}(y)\|_{\mathcal{H}} &\leq \| J_{\gamma_1 A}(z_2) - J_{\gamma_1 A}(z_1) \|_{\mathcal{H}} + \| J_{\gamma_1 A}(z_2) - J_{\gamma_2 A}(z_2) \|_{\mathcal{H}} \nonumber \\
&\leq \| z_1 - z_2 \|_{\mathcal{H}} + \frac{|\gamma_1 - \gamma_2|}{\gamma_1} \| z_2 - J_{\gamma_1 A}(z_2) \|_{\mathcal{H}}.
\label{eq:proof_T_diff_step}
\end{align}
Utilizing the identities $z_{i} = y - \gamma_{i} B(y)$ and $\lambda_{1} T_{\lambda_{1}, \gamma_{1}}(y) = y - J_{\gamma_{1} A}(z_{1})$, it follows that
\begin{align}
\left\| z_{2} - J_{\gamma_{1} A}(z_{2}) \right\|_{\mathcal{H}} 
&\le \left\| \lambda_{1} T_{\lambda_{1}, \gamma_{1}}(y) \right\|_{\mathcal{H}} + \gamma_{1} \| B(y) \|_{\mathcal{H}} + 2 \| z_{1} - z_{2} \|_{\mathcal{H}} \nonumber \\
&= \left\| \lambda_{1} T_{\lambda_{1}, \gamma_{1}}(y) \right\|_{\mathcal{H}} + \left( \gamma_{1} + 2 \left| \gamma_{1} - \gamma_{2} \right| \right) \| B(y) \|_{\mathcal{H}}, \label{eq:residual_argument_chain}
\end{align}
where the first inequality follows from the triangle inequality and the non-expansiveness of $J_{\gamma_1 A}$ (i.e., $\| J_{\gamma_1 A}(z_1) - J_{\gamma_1 A}(z_2) \|_{\mathcal{H}} \le \| z_1 - z_2 \|_{\mathcal{H}}$), and the final equality utilizes the identity $\|z_1 - z_2\|_{\mathcal{H}} = |\gamma_1 - \gamma_2| \|B(y)\|_{\mathcal{H}}$.

\noindent By substituting \eqref{eq:residual_argument_chain} into \eqref{eq:proof_T_diff_step}, we obtain
\begin{align}
\|\lambda_1 T_{\lambda_1, \gamma_1}(y) - \lambda_2 T_{\lambda_2, \gamma_2}(y)\|_{\mathcal{H}} &\le \frac{|\gamma_1 - \gamma_2|}{\gamma_1} \left\| \lambda_1 T_{\lambda_1, \gamma_1}(y) \right\|_{\mathcal{H}} + \left( 2 + \frac{2|\gamma_1 - \gamma_2|}{\gamma_1} \right) |\gamma_1 - \gamma_2| \|B(y)\|_{\mathcal{H}} \nonumber \\
&\le \frac{|\gamma_1 - \gamma_2|}{\gamma_1} \left( \|\lambda_1 T_{\lambda_1, \gamma_1}(x)\|_{\mathcal{H}} + \frac{\lambda_1}{\Lambda_1} \|x - y\|_{\mathcal{H}} \right) \nonumber \\
&\quad + \left( 2 + \frac{2|\gamma_1 - \gamma_2|}{\gamma_1} \right) |\gamma_1 - \gamma_2| \left( \|B(x)\|_{\mathcal{H}} + \frac{1}{\nu} \|x - y\|_{\mathcal{H}} \right), \label{eq:y_residual_combined_chain}
\end{align}
where the second inequality follows from the Lipschitz continuity $\|B(y)\|_{\mathcal{H}} \leq \|B(x)\|_{\mathcal{H}} + \frac{1}{\nu} \|x - y\|_{\mathcal{H}}$ and $\|\lambda_1 T_{\lambda_1, \gamma_1}(y)\|_{\mathcal{H}} \leq \|\lambda_1 T_{\lambda_1, \gamma_1}(x)\|_{\mathcal{H}} + \frac{\lambda_1}{\Lambda_1} \|x - y\|_{\mathcal{H}}$.

\noindent Further, recalling that $\bar{x} \in \operatorname{zer}(A+B) \implies \lambda_1 T_{\lambda_1, \gamma_1}(\bar{x}) = 0$ and applying triangle inequality yields:
\begin{align}
\|\lambda_1 T_{\lambda_1, \gamma_1}(x)\|_{\mathcal{H}} 
&\le \| x - \bar{x} \|_{\mathcal{H}} + \left\| J_{\gamma_1 A} \left( x - \gamma_1 Bx \right) - J_{\gamma_1 A} \left( \bar{x} - \gamma_1 B\bar{x} \right) \right\|_{\mathcal{H}} \nonumber \\
&\le \| x - \bar{x} \|_{\mathcal{H}} + \left\| \left( x - \gamma_1 Bx \right) - \left( \bar{x} - \gamma_1 B\bar{x} \right) \right\|_{\mathcal{H}} \nonumber \\
&\le 2 \| x - \bar{x} \|_{\mathcal{H}}, \label{eq:pure_non_expansive_bound}
\end{align}
where the non-expansiveness of $J_{\gamma_1 A}$ and $\operatorname{Id} - \gamma_1 B$ is ensured by Proposition \ref{prop:max_comonotone_properties} (2) and Proposition \ref{prop:cocoercive_averaged}.

\noindent Substituting \eqref{eq:proof_T_lip_term}, \eqref{eq:y_residual_combined_chain}, and \eqref{eq:pure_non_expansive_bound} into \eqref{eq:proof_T_split_final} and rearranging the terms, it follows that:
\begin{align}
\|\lambda_{1} T_{\lambda_{1}, \gamma_{1}}(x) - \lambda_{2} T_{\lambda_{2}, \gamma_{2}}(y)\|_{\mathcal{H}} &\le \frac{\lambda_1}{\Lambda_1} \|x - y\|_{\mathcal{H}} + \left\| \lambda_{1} T_{\lambda_{1}, \gamma_{1}}(y) - \lambda_{2} T_{\lambda_{2}, \gamma_{2}}(y) \right\|_{\mathcal{H}} \nonumber \\
&\le \underbrace{\left( \frac{\lambda_1}{\Lambda_1} + \frac{\lambda_1}{\Lambda_1} \frac{|\gamma_1 - \gamma_2|}{\gamma_1} + \frac{2|\gamma_1 - \gamma_2|}{\nu} + \frac{2|\gamma_1 - \gamma_2|^2}{\gamma_1 \nu} \right)}_{:= \mathcal{C}_x} \|x - y\|_{\mathcal{H}} \nonumber \\
&\quad + \left( \frac{2 |\gamma_1 - \gamma_2|}{\gamma_1} \right) \|x - \bar{x}\|_{\mathcal{H}} + \underbrace{\left( 2|\gamma_1 - \gamma_2| + \frac{2 |\gamma_1 - \gamma_2|^2}{\gamma_1} \right)}_{:= \mathcal{C}_B} \|B(x)\|_{\mathcal{H}}. \label{eq:full_sensitivity_chain}
\end{align}
By imposing the constraints $\gamma_{i} \in (0, 2\nu)$, which imply $|\gamma_{1} - \gamma_{2}| < 2\nu$ and $1 < 2\nu/\gamma_{i}$, the coefficients $C_{B}$ and $C_{x}$ admit uniform upper bounds. Noting that $2\gamma_1 + 2|\gamma_1 - \gamma_2| < 8\nu$, we obtain:
\begin{equation}\label{eq:coeff_bounds}
C_{B} < \frac{8\nu |\gamma_{1} - \gamma_{2}|}{\gamma_{1}}, \quad \text{and} \quad C_{x} < \frac{4\nu}{\gamma_{1}} \frac{\lambda_{1}}{\Lambda_{1}} + \frac{8 |\gamma_{1} - \gamma_{2}|}{\gamma_{1}} \le 4 \left( \frac{\lambda_{1}}{\Lambda_{1}} + 4 \right) \frac{\nu}{\gamma_{1}}.
\end{equation}
Substituting \eqref{eq:coeff_bounds} into \eqref{eq:full_sensitivity_chain} validates the desired result (ii).

To establish assertion (iii), we consider the time derivative of the operator defined as the limit of its difference quotient. By substituting $x = x(t+h)$, $y = x(t)$, $\gamma_1 = \gamma(t+h)$, and $\gamma_2 = \gamma(t)$ into \eqref{eq:asymptotic_consistent_form}, we formulate the following estimates for almost every $t \geq t_{0}$:
\begin{align*}
\left\| \frac{d}{dt} \left( \lambda(t) T_{\lambda(t), \gamma(t)}(x(t)) \right) \right\|_{\mathcal{H}} &= \lim_{h \to 0^+} \frac{1}{h} \left\| \lambda(t+h) T_{\lambda(t+h), \gamma(t+h)}(x(t+h)) - \lambda(t) T_{\lambda(t), \gamma(t)}(x(t)) \right\|_{\mathcal{H}} \nonumber \\
&\le \lim_{h \to 0^+} \Bigg[ 4 \left( \frac{\lambda(t+h)}{\Lambda(t+h)} + 4 \right) \frac{\nu}{\gamma(t+h)} \frac{\| x(t+h) - x(t) \|_{\mathcal{H}}}{h} \nonumber \\
&\quad + \frac{2}{\gamma(t+h)} \frac{| \gamma(t+h) - \gamma(t) |}{h} \| x(t+h) - \bar{x} \|_{\mathcal{H}} \nonumber \\
&\quad + \frac{8\nu}{\gamma(t+h)} \frac{| \gamma(t+h) - \gamma(t) |}{h} \| B(x(t+h)) \|_{\mathcal{H}} \Bigg]\nonumber \\
&= 4 \left( \frac{\lambda(t)}{\Lambda(t)} + 4 \right) \frac{\nu}{\gamma(t)} \|\dot{x}(t)\|_{\mathcal{H}} + \frac{2 |\dot{\gamma}(t)|}{\gamma(t)} \| x(t) - \bar{x} \|_{\mathcal{H}} + \frac{8\nu |\dot{\gamma}(t)|}{\gamma(t)} \| B(x(t)) \|_{\mathcal{H}} \nonumber \\
&\le 4 \left( \frac{1}{\underline{\Omega}} + 4 \right) \frac{\nu}{\gamma(t)} \|\dot{x}(t)\|_{\mathcal{H}} + \frac{2 |\dot{\gamma}(t)|}{\gamma(t)} \| x(t) - \bar{x} \|_{\mathcal{H}} + \frac{8\nu |\dot{\gamma}(t)|}{\gamma(t)} \| B(x(t)) \|_{\mathcal{H}},
\end{align*}
which establishes assertion (iii).
\end{proof}

\subsection{Prior Estimate of Lyapunov Functional}
For any $\bar{x} \in \operatorname{zer}(A+B)$ and a strong global solution $x(t)$ of system \eqref{eq:SHDS}, we introduce the generalized energy functional $\mathcal{E}(t):$
\begin{equation}
\label{eq:energy_functional_01}
\mathcal{E}(t) := \frac{1}{2} \left\| a(t) \left( x(t) - \bar{x} \right) + m(t) v(t) \right\|_{\mathcal{H}}^2 + \frac{c(t)}{2} \|x(t) - \bar{x}\|_{\mathcal{H}}^2,
\end{equation}
where $v(t) := \dot{x}(t) + \beta(t) T_{\lambda(t), \gamma(t)}(x(t))$, and $a(t), m(t), c(t)$ are positive continuous and differentiable functions on $[t_0, +\infty)$ to be appropriately adjusted.
\vskip 1mm
Also, we define auxiliary functions: $\mathcal{H}(t) := b(t) - \dot{\beta}(t) - \alpha(t)\beta(t)$, $\mathcal{Q}_1(t) := a(t) \left( a(t)\beta(t) + m(t)\mathcal{H}(t) \right) + c(t)\beta(t)$, $\mathcal{Q}_2(t) := m(t) \left( a(t)\beta(t) + m(t)\mathcal{H}(t) \right)$, $\mathcal{W}_v(t) := -m(t) \left( a(t) + \dot{m}(t) - m(t)\alpha(t) \right)$, and $\mathcal{W}_T(t) := \mathcal{Q}_1(t)\Lambda(t) - \frac{\mathcal{Q}_2^2(t)}{2 \mathcal{W}_v(t)}$.

\begin{theorem}
\label{thm:m_weight_bounds}
Assume that Assumptions {\rm {\bf (H1)-(H3)}} hold and let $x(\cdot):[t_0, +\infty) \to \mathcal{H}$ be a solution trajectory of system \eqref{eq:SHDS}. Suppose that the following conditions are satisfied for almost every $t \geq t_0$:
\begin{equation*}
\begin{aligned}
    &\text{(C1)} \quad m(t) \dot{a}(t) + a(t) \left( a(t) + \dot{m}(t) - m(t) \alpha(t) \right) + c(t) = 0; \quad \text{(C2)} \quad \mathcal{W}_v(t) > 0, \, \mathcal{W}_T(t) > 0;\\ \quad 
    &\text{(C3)} \quad a(t)\dot{a}(t) + \frac{1}{2}\dot{c}(t) \le 0; \quad
    \text{(C4)} \quad \inf_{t \ge t_0} c(t) = \underline{c} > 0.
\end{aligned}
\end{equation*}
Then, the functional $\mathcal{E}(t)$ is non-increasing on $[t_0, +\infty)$ and the following properties hold:
\begin{enumerate}[label=(\roman*)]
    \item $\sup_{t \geq t_0} \|x(t) - \bar{x}\|_{\mathcal{H}}< +\infty$;
    \item $\sup_{t \geq t_0} \left\| a(t) \left( x(t) - \bar{x} \right) + m(t) \left( \dot{x}(t) + \beta(t) T_{\lambda(t), \gamma(t)}(x(t)) \right) \right\|_{\mathcal{H}} < +\infty$;
    \item $\int_{t_0}^{+\infty} \mathcal{W}_T(t) \| T_{\lambda(t), \gamma(t)}(x(t)) \|_{\mathcal{H}}^2 dt < +\infty$;
    \item $\int_{t_0}^{+\infty} \mathcal{W}_v(t) \| \dot{x}(t) + \beta(t) T_{\lambda(t), \gamma(t)}(x(t)) \|_{\mathcal{H}}^2 dt < +\infty$.
\end{enumerate}
\end{theorem}

\begin{proof}
Differentiating the Lyapunov function $\mathcal{E}(t)$ defined in \eqref{eq:energy_functional_01} yields
\begin{equation}\label{eq:proof_m_dotE}
\dot{\mathcal{E}}(t) = \left\langle a(t)(x(t) - \bar{x}) + m(t) v(t), \; \frac{d}{dt} \left( a(t)(x(t) - \bar{x}) + m(t) v(t) \right) \right\rangle_{\mathcal{H}} + \frac{1}{2}\dot{c}(t)\|x(t) - \bar{x}\|_{\mathcal{H}}^2 + c(t) \langle x(t) - \bar{x}, \dot{x}(t) \rangle_{\mathcal{H}}.
\end{equation}
Along the trajectories of system \eqref{eq:SHDS}, the time derivative of the composite vector field expands as:
\begin{align}
\frac{d}{dt} \left[ a(t) ( x(t) - \bar{x} ) + m(t) v(t) \right] &= \dot{a}(t) ( x(t) - \bar{x} ) + a(t)\dot{x}(t) + \dot{m}(t)v(t) + m(t)\dot{v}(t) \nonumber \\
&= \dot{a}(t) ( x(t) - \bar{x} ) + ( a(t) + \dot{m}(t) - m(t)\alpha(t) )v(t) \nonumber \\
&\quad - ( a(t)\beta(t) + m(t)\mathcal{H}(t) ) T_{\lambda(t), \gamma(t)}(x(t)). \label{eq:m_vector_field}
\end{align}
Substituting \eqref{eq:m_vector_field} and $\dot{x}(t) = v(t) - \beta(t)T_{\lambda(t), \gamma(t)}(x(t))$ into \eqref{eq:proof_m_dotE}, and regrouping terms, we obtain:
\begin{align}
\dot{\mathcal{E}}(t) =\;& \left( a(t)\dot{a}(t) + \frac{1}{2}\dot{c}(t) \right) \|x(t) - \bar{x}\|_{\mathcal{H}}^2 + m(t) \left( a(t) + \dot{m}(t) - m(t)\alpha(t) \right) \|v(t)\|_{\mathcal{H}}^2 \nonumber \\
& + \left( a(t) \left( a(t) + \dot{m}(t) - m(t)\alpha(t) \right) + m(t)\dot{a}(t) + c(t) \right) \langle x(t) - \bar{x}, v(t) \rangle_{\mathcal{H}} \nonumber \\
& - \mathcal{Q}_1(t) \langle x(t) - \bar{x}, T_{\lambda(t), \gamma(t)}(x(t)) \rangle_{\mathcal{H}} - \mathcal{Q}_2(t) \langle v(t), T_{\lambda(t), \gamma(t)}(x(t)) \rangle_{\mathcal{H}}. \label{eq:proof_m_expanded}
\end{align}
According to condition (C1), the term $\langle x(t) - \bar{x}, v(t) \rangle_{\mathcal{H}}$ vanishes. Assumption {\bf (H1)} establishes $\Lambda(t) > 0$, which, in conjunction with condition $\mathrm{(C2)}$, implies $\mathcal{Q}_1(t) > 0$. By virtue of $\Lambda(t)$-cocoercivity of $T_{\lambda(t), \gamma(t)}$, it holds that $\mathcal{Q}_1(t) \langle x(t) - \bar{x}, T_{\lambda(t), \gamma(t)}(x(t)) \rangle_{\mathcal{H}} \ge \mathcal{Q}_1(t) \Lambda(t)\|T_{\lambda(t), \gamma(t)}(x(t))\|_{\mathcal{H}}^2$. Additionally, employing Young's inequality for the remaining term, we have:
\begin{equation}\label{eq:m_young}
-\mathcal{Q}_2(t) \langle v(t), T_{\lambda(t), \gamma(t)}(x(t)) \rangle_{\mathcal{H}} \le \frac{\mathcal{W}_v(t)}{2} \|v(t)\|_{\mathcal{H}}^2 + \frac{\mathcal{Q}_2^2(t)}{2 \mathcal{W}_v(t)} \|T_{\lambda(t), \gamma(t)}(x(t))\|_{\mathcal{H}}^2.
\end{equation}
Substituting \eqref{eq:m_young} into \eqref{eq:proof_m_expanded} and using the definitions of $\mathcal{W}_v(t)$ and $\mathcal{W}_T(t)$, the derivative of the Lyapunov functional simplifies to:
\begin{align}
\dot{\mathcal{E}}(t) \le \;& \left( a(t)\dot{a}(t) + \frac{1}{2}\dot{c}(t) \right)\|x(t) - \bar{x}\|_{\mathcal{H}}^2 - \frac{\mathcal{W}_v(t)}{2}\|v(t)\|_{\mathcal{H}}^2 - \mathcal{W}_T(t)\|T_{\lambda(t), \gamma(t)}(x(t))\|_{\mathcal{H}}^2. \label{eq:proof_m_dotE_final}
\end{align}
Under conditions (C2) and (C3), the non-positivity of every term on the right-hand side of \eqref{eq:proof_m_dotE_final} ensures that $\dot{\mathcal{E}}(t) \le 0$ for almost every $t \ge t_0$, which yields $\left\| a(t) \left( x(t) - \bar{x} \right) + m(t) \left( \dot{x}(t) + \beta(t) T_{\lambda(t), \gamma(t)}(x(t)) \right) \right\|_{\mathcal{H}}^2 \le 2\mathcal{E}(t) \le 2\mathcal{E}(t_0)$. Given condition (C4), the boundedness of the trajectory $x(t)$ immediately follows from $\underline{c}\|x(t) - \bar{x}\|_{\mathcal{H}}^2 \le 2\mathcal{E}(t) \le 2\mathcal{E}(t_0)$. Furthermore, integrating the inequality \eqref{eq:proof_m_dotE_final} over $[t_0, +\infty)$ yields the desirable estimates in (iii) and (iv).
\end{proof}

\subsection{Asymptotic analysis and Improved Convergence Rates}
In this section, we present our main theoretical results concerning the long-term behavior of system \eqref{eq:SHDS}. The following theorem provides a unified characterization, ranging from qualitative trajectory regularity to accelerated convergence rates.

\begin{theorem}
\label{thm:pointwise_magnitudes_layered}
Suppose Assumptions {\bf (H1)--(H3)} and conditions {\rm (C1)--(C4)} hold. Let $x(t)$ be a solution trajectory of system \eqref{eq:SHDS}. As $t \to +\infty$, the following assertions hold:

\smallskip
\noindent \textbf{(1) Asymptotic Magnitudes.} If {\rm (C5)} $\bar{a} := \sup_{t \ge t_0} a(t) < +\infty$ and $\sup_{t \geq t_0} \frac{m(t) \beta(t)}{\lambda(t)} < +\infty$ are satisfied, then:
\begin{enumerate}[label=(\roman*), nosep]
    \item $\|T_{\lambda(t), \gamma(t)}(x(t))\|_{\mathcal{H}} = \mathcal{O}( \frac{1}{\lambda(t)})$;
    \item $\|\dot{x}(t)\|_{\mathcal{H}} = \mathcal{O}( \frac{1}{m(t)} )$.
\end{enumerate}

\smallskip
\noindent \textbf{(2) Operator Derivative Magnitude.} If, in addition to {\rm (C5)}, parameters fulfill {\rm (C6)} $0<\underline{m}' \leq \dot{m}(t) \leq \bar{m}'$, $m(t) \geq \dot{m}(t)$, $\underline{\gamma} := \inf_{t \geq t_0} \gamma(t) > 0$; and {\rm (C7)} $\frac{|\dot{\gamma}(t)|}{\gamma(t)}, \frac{|\dot{\lambda}(t)|}{\lambda(t)} = \mathcal{O}( \frac{\dot{m}(t)}{m(t)} )$, then:
\begin{enumerate}[label=(\roman*), resume, nosep]
    \item $\| \frac{d}{dt} T_{\lambda(t), \gamma(t)}(x(t)) \|_{\mathcal{H}} = \mathcal{O}\left( \frac{\dot{m}(t)}{m(t) \lambda(t)} \right)$.
\end{enumerate}

\smallskip
\noindent \textbf{(3) Accelerated Trajectory and Residual Rates.} Building upon conditions in \textbf{(1)} and \textbf{(2)}, if parameters satisfy {\rm (C8)} $\alpha(t) = \mathcal{O}( \frac{\dot{m}(t)}{m(t)} )$, $\beta(t) = \mathcal{O}( \frac{m(t)}{\dot{m}(t)} )$, $b(t) = \mathcal{O}(1)$, $b^2(t) \leq C_b \frac{\dot{m}(t) \mathcal{W}_T(t)}{m^3(t)}$; {\rm (C9)} $\lambda(t) \ge C_{\lambda} m^2(t)$; {\rm (C10)} $\mathcal{W}_v(t) \ge C_v m(t)\dot{m}(t)$, $\mathcal{W}_T(t) \ge C_T m^2(t)\dot{m}(t)$ for some $C_b, C_{\lambda}, C_v, C_T > 0$; and {\rm (C11)} $\liminf_{t \to +\infty} \frac{m(t) \mathcal{W}_T(t)}{\dot{m}(t) \lambda^2(t)}:=w > 0$, then:
\begin{enumerate}[label=(\roman*), resume, nosep, itemsep=2pt]
    \item $\|\ddot{x}(t)\|_{\mathcal{H}} = \mathcal{O}\left( \frac{\dot{m}(t)}{m^2(t)} \right)$;
    \item $\|\dot{x}(t)\|_{\mathcal{H}} = o\left( \frac{1}{m(t)} \right)$;
    \item $\|T_{\lambda(t), \gamma(t)}(x(t))\|_{\mathcal{H}} = o\left( \frac{1}{\lambda(t)} \right)$;
    \item $\| A_{\gamma(t)} ( x(t) - \gamma(t) Bx(t) ) + Bx(t) \|_{\mathcal{H}} = o\left( \frac{1}{\gamma(t)} \right)$;
    \item $\| \frac{d}{dt} ( A_{\gamma(t)} ( x(t) - \gamma(t) Bx(t) ) + Bx(t) ) \|_{\mathcal{H}} = \mathcal{O}\left( \frac{\dot{m}(t)}{m(t) \gamma(t)} \right) + o\left( \frac{\lambda(t) | \frac{d}{dt} ( \frac{\gamma(t)}{\lambda(t)} ) |}{\gamma^2(t)} \right)$.
\end{enumerate}
\end{theorem}

\begin{proof}
\hfill \\
\noindent \textbf{(1) Basic Asymptotic Magnitudes:} By combining the $\frac{1}{\Lambda(t)}>0$-Lipschitz continuity of $T_{\lambda(t), \gamma(t)}$ with the boundedness of the trajectory $\|x(t) - \bar{x}\|_{\mathcal{H}} \le M_x$ established in Theorem \ref{thm:m_weight_bounds}, assertion (i) follows from the estimate:
\begin{equation}
\label{T__first_rate}
\|T_{\lambda(t), \gamma(t)}(x(t))\|_{\mathcal{H}} \le \frac{1}{\lambda(t) \Omega(t)} \|x(t) - \bar{x}\|_{\mathcal{H}} \le \frac{M_x}{\underline{\Omega} \lambda(t)} = \mathcal{O}\left( \frac{1}{\lambda(t)} \right),
\end{equation}
where $\Omega(t):=\frac{4\nu \left( \rho + \gamma(t) \right) - \gamma^2(t)}{4\gamma(t) \left( \nu + \rho \right)}$ and $\underline{\Omega} := \inf_{t \ge t_0} \Omega(t)> 0$ from Lemma \ref{prop:omega_uniform_positivity}. Utilizing Theorem \ref{thm:m_weight_bounds} (ii), condition (C5), \eqref{T__first_rate} and triangle inequality, we have:
\begin{align*}
\|m(t) \dot{x}(t)\|_{\mathcal{H}} &\le \left\| a(t) \left( x(t) - \bar{x} \right) + m(t) \left( \dot{x}(t) + \beta(t) T_{\lambda(t), \gamma(t)}(x(t)) \right) \right\|_{\mathcal{H}} \nonumber \\
&\quad + a(t)\|x(t) - \bar{x}\|_{\mathcal{H}} + m(t)\beta(t)\|T_{\lambda(t), \gamma(t)}(x(t))\|_{\mathcal{H}} \nonumber \\
&\le \sqrt{2\mathcal{E}(t_0)} + \bar{a} M_x + \left( \frac{M_x}{\underline{\Omega}} \right) \left( \frac{m(t)\beta(t)}{\lambda(t)} \right)\nonumber \\
&< + \infty,
\label{eq:m_isolation_refined}
\end{align*}
which yields $\|\dot{x}(t)\|_{\mathcal{H}} = \mathcal{O}\left( \frac{1}{m(t)} \right)$, thus establishing assertion (ii).

\noindent \textbf{ (2) Operator Derivative Magnitude:}
For assertion (iii), by combining Lemma \ref{lem:sharp_sensitivity_final} (iii), the boundedness of the trajectory, the velocity magnitude from Theorem \ref{thm:pointwise_magnitudes_layered} (ii), and conditions (C6)--(C7), along with the Lipschitz continuity of $B$, it follows that as $t \to +\infty$:

\begin{align}
\left\| \frac{d}{dt}\lambda(t)T_{\lambda(t),\gamma(t)}(x(t)) \right\|_{\mathcal{H}} &\le  \frac{4 \left( \frac{1}{\underline{\Omega}} + 4 \right)\nu}{\gamma(t)} \left\| \dot{x}(t) \right\|_{\mathcal{H}} + \frac{2 |\dot{\gamma}(t)|}{\gamma(t)} \| x(t) - \bar{x} \|_{\mathcal{H}} + \frac{8\nu |\dot{\gamma}(t)|}{\gamma(t)} \| B(x(t)) \|_{\mathcal{H}} \nonumber \\
&\le \frac{4 \left( \frac{1}{\underline{\Omega}} + 4 \right)\nu}{\underline{\gamma} \underline{m}'} \mathcal{O}\left( \frac{\dot{m}(t)}{m(t)} \right) +\frac{2 |\dot{\gamma}(t)|}{\gamma(t)} \| x(t) - \bar{x} \|_{\mathcal{H}} + \frac{8\nu |\dot{\gamma}(t)|}{\gamma(t)} \| B(x(t)) \|_{\mathcal{H}} \nonumber \\
&= \mathcal{O}\left( \frac{\dot{m}(t)}{m(t)} \right).
\end{align}

\noindent By applying the product rule and triangle inequality to $T_{\lambda(t), \gamma(t)}(x(t)) = \lambda^{-1}(t) \left( \lambda(t) T_{\lambda(t), \gamma(t)}(x(t)) \right)$, we obtain:
\begin{align}
\left\| \frac{d}{dt} T_{\lambda(t), \gamma(t)}(x(t)) \right\|_{\mathcal{H}} &\le \frac{1}{\lambda(t)} \left\| \frac{d}{dt} \left( \lambda(t) T_{\lambda(t), \gamma(t)}(x(t)) \right) \right\|_{\mathcal{H}} + \frac{|\dot{\lambda}(t)|}{\lambda^2(t)} \left\| \lambda(t) T_{\lambda(t), \gamma(t)}(x(t)) \right\|_{\mathcal{H}} \nonumber \\
&\le \frac{1}{\lambda(t)} \left( \mathcal{O}\left( \frac{\dot{m}(t)}{m(t)} \right) + \mathcal{O}\left( \frac{\dot{m}(t)}{m(t)} \right) \cdot \mathcal{O}(1) \right) \nonumber \\
&= \mathcal{O}\left( \frac{\dot{m}(t)}{m(t) \lambda(t)} \right), \label{eq:dotT_m_final_result}
\end{align}
where we have utilized \eqref{T__first_rate}, condition {\rm (C7)}, and Lemma~\ref{lem:sharp_sensitivity_final}(iii).

\noindent \textbf{(3) Improved Trajectory and Residual Rates:}
For assertion (iv), by incorporating system \eqref{eq:SHDS}, Theorem \ref{thm:pointwise_magnitudes_layered} (i)--(iii), and conditions (C8)--(C9), it follows that as $t \to +\infty$:
\begin{align}
\|\ddot{x}(t)\|_{\mathcal{H}} &\le |\alpha(t)| \|\dot{x}(t)\|_{\mathcal{H}} + |\beta(t)| \left\| \dot{T}_{\lambda(t), \gamma(t)}(x(t)) \right\|_{\mathcal{H}} + |b(t)| \|T_{\lambda(t), \gamma(t)}(x(t))\|_{\mathcal{H}} \nonumber \\
&\le \mathcal{O}\left( \frac{\dot{m}(t)}{m(t)} \right) \mathcal{O}\left( \frac{1}{m(t)} \right) + \mathcal{O}\left( \frac{m(t)}{\dot{m}(t)} \right) \mathcal{O}\left( \frac{\dot{m}(t)}{m(t)\lambda(t)} \right) + \mathcal{O}\left(1\right)\mathcal{O}\left(\frac{1}{\lambda(t)}\right)\nonumber \\
&= \mathcal{O}\left( \frac{\dot{m}(t)}{m^2(t)} \right) + \mathcal{O}\left(\frac{1}{m^2(t)}\right)  \nonumber \\
&= \mathcal{O}\left( \frac{\dot{m}(t)}{m^2(t)} \right).
\label{eq:accel_magnitude_step_O_refined}
\end{align}

Now, we focus on improving the convergence rate of $\|\dot{x}(t)\|_{\mathcal{H}}$ to validate assertion (v). Let $k_{\beta}>0$ be such that $\beta^2(t)\le k_{\beta}^2m^2(t)/\dot m^2(t)$ for all sufficiently large $t$, which follows from condition {\rm (C8)}. By the decomposition $\dot{x}(t)=v(t)-\beta(t)T_{\lambda(t),\gamma(t)}(x(t))$, Theorem~\ref{thm:m_weight_bounds}(iii)--(iv), and conditions {\rm (C10)}--{\rm (C11)}, we have
\begin{align}
\int_{t_0}^{+\infty}m(t)\dot m(t)\|\dot{x}(t)\|_{\mathcal{H}}^2dt
&\le 2\int_{t_0}^{+\infty}m(t)\dot m(t)\|v(t)\|_{\mathcal{H}}^2dt \nonumber\\
&\quad +2\int_{t_0}^{+\infty}m(t)\dot m(t)\beta^2(t)\|T_{\lambda(t),\gamma(t)}(x(t))\|_{\mathcal{H}}^2dt \nonumber\\
&\le \frac{2}{C_v}\int_{t_0}^{+\infty}\mathcal{W}_v(t)\|v(t)\|_{\mathcal{H}}^2dt \nonumber\\
&\quad +K_0\int_{t_0}^{+\infty}\mathcal{W}_T(t)\|T_{\lambda(t),\gamma(t)}(x(t))\|_{\mathcal{H}}^2dt<+\infty,
\label{eq:velocity_log_integral}
\end{align}
where $K_0:=\frac{2k_{\beta}^2}{K_1(\underline m')^2}>0$. Indeed, the last term is controlled by $\mathcal{W}_T(t)$ because condition {\rm (C11)} gives $\mathcal{W}_T(t)\ge \frac{w}{2}\frac{\dot m(t)\lambda^2(t)}{m(t)}$ for sufficient large $t$, and condition {\rm (C9)} together with $\dot m(t)\ge\underline m'>0$ yields $\mathcal{W}_T(t)\ge K_1m^3(t)\dot m(t)$ for some  $K_1:=\frac{wC_\lambda^2}{2}>0$.

Set $F(t):=m(t)\dot{x}(t)$ and $r(t):=\dot m(t)/m(t)$. Then $F$ is locally absolutely continuous. Since $\dot m(t)\ge\underline m'>0$, one has $m(t)\to+\infty$ and $\int_{t_0}^{+\infty}r(t)dt=+\infty$. The estimate \eqref{eq:velocity_log_integral} is equivalent to $\int_{t_0}^{+\infty}r(t)\|F(t)\|_{\mathcal{H}}^2dt<+\infty$.
Moreover, by the assertion {\rm (ii)} of Theorem~\ref{thm:pointwise_magnitudes_layered} and \eqref{eq:accel_magnitude_step_O_refined}, we have
\[
\|\dot F(t)\|_{\mathcal{H}}
\le \dot m(t)\|\dot x(t)\|_{\mathcal{H}}+m(t)\|\ddot x(t)\|_{\mathcal{H}}
=\mathcal{O}\left(\frac{\dot m(t)}{m(t)}\right)
=\mathcal{O}(r(t)).
\]
Thus, utilizing Lemma~\ref{lem:weighted_decay_criterion} gives $m(t)\dot{x}(t)\to0$ as $t \to +\infty$, which proves assertion (v).

For assertion (vi), given condition (C11), there exists $t_2 \ge t_0$ such that for all $t \ge t_2$, $\frac{m(t) \mathcal{W}_T(t)}{\dot{m}(t) \lambda^2(t)} \ge \frac{w}{2} > 0$. 
From Theorem \ref{thm:m_weight_bounds} (iii), it follows that
\begin{align}
\frac{w}{2} \int_{t_2}^{+\infty} \frac{\dot{m}(t)}{m(t)} \left\| \lambda(t) T_{\lambda(t), \gamma(t)}(x(t)) \right\|_{\mathcal{H}}^2 dt &\le \int_{t_2}^{+\infty} \left( \frac{m(t) \mathcal{W}_T(t)}{\dot{m}(t) \lambda^2(t)} \right) \frac{\dot{m}(t)}{m(t)} \left\| \lambda(t) T_{\lambda(t), \gamma(t)}(x(t)) \right\|_{\mathcal{H}}^2 dt \nonumber \\
&= \int_{t_2}^{+\infty} \mathcal{W}_T(t) \left\| T_{\lambda(t), \gamma(t)}(x(t)) \right\|_{\mathcal{H}}^2 dt < +\infty, 
\label{eq:U_integral_log_m}
\end{align}
which implies that $\int_{t_0}^{+\infty} \frac{\dot{m}(t)}{m(t)} \left\| \lambda(t) T_{\lambda(t), \gamma(t)}(x(t)) \right\|_{\mathcal{H}}^2 dt < +\infty$. Let $U(t):=\lambda(t)T_{\lambda(t),\gamma(t)}(x(t))$ and $r(t):=\frac{\dot m(t)}{m(t)}$, then the preceding estimate~\eqref{eq:U_integral_log_m} is precisely $\int_{t_0}^{+\infty}r(t)\|U(t)\|_{\mathcal{H}}^2dt<+\infty$. Furthermore, as shown in the proof of assertion {\rm (iii)}, $U$ is locally absolutely continuous and $\|\dot U(t)\|_{\mathcal{H}}=\mathcal{O}(r(t))$. Therefore, applying Lemma~\ref{lem:weighted_decay_criterion} to $U(t)$ yields $U(t)\to0$, namely $\lambda(t)\|T_{\lambda(t),\gamma(t)}(x(t))\|_{\mathcal{H}}\to0$ as $t \to +\infty$. This validates assertion (vi).

To prove the remain results, invoking the rate $\|T_{\lambda(t), \gamma(t)}(x(t))\|_{\mathcal{H}} = o(\frac{1}{\lambda(t)})$ established in Theorem \ref{thm:pointwise_magnitudes_layered} (vi), we get
\begin{align*}
\left\| A_{\gamma(t)} \left( x(t) - \gamma(t) Bx(t) \right) + Bx(t) \right\|_{\mathcal{H}} &= \frac{\lambda(t)}{\gamma(t)} \left\| T_{\lambda(t), \gamma(t)}(x(t)) \right\|_{\mathcal{H}}  = o\left( \frac{1}{\gamma(t)} \right),
\label{eq:operator_sum_magnitude_result}
\end{align*}
which establishes assertion (vii). 

\noindent By applying triangle inequality to the time derivative identity of the regularized operator sum and utilizing the results {\rm (iii)} and {\rm (vi)} from Theorem \ref{thm:pointwise_magnitudes_layered}, we verify assertion {\rm (viii)} through the following estimation:
\begin{equation*}
\begin{aligned}
\left\|\frac{d}{dt}\left(A_{\gamma(t)}\left(x(t)-\gamma(t)Bx(t)\right)+Bx(t)\right)\right\|_{\mathcal{H}} 
&\le \frac{\lambda(t)}{\gamma^2(t)} \left| \frac{d}{dt} \left( \frac{\gamma(t)}{\lambda(t)} \right) \right| \left\| \lambda(t) T_{\lambda(t), \gamma(t)}(x(t)) \right\|_{\mathcal{H}} + \frac{\lambda(t)}{\gamma(t)} \left\| \dot{T}_{\lambda(t), \gamma(t)}(x(t)) \right\|_{\mathcal{H}} \\
&\le o\left( \frac{\lambda(t) \left| \frac{d}{dt} \left( \frac{\gamma(t)}{\lambda(t)} \right) \right|}{\gamma^2(t)} \right) + \frac{\lambda(t)}{\gamma(t)} \mathcal{O}\left( \frac{\dot{m}(t)}{m(t) \lambda(t)} \right) \\
&= \mathcal{O}\left( \frac{\dot{m}(t)}{m(t) \gamma(t)} \right) + o\left( \frac{\lambda(t) \left| \frac{d}{dt} \left( \frac{\gamma(t)}{\lambda(t)} \right) \right|}{\gamma^2(t)} \right).
\end{aligned}
\end{equation*}
\end{proof}

\subsection{Weak Convergence of the Trajectories}
\label{sec:weak_convergence}

In this section, we establish that the trajectory $x(t)$ generated by system \eqref{eq:SHDS} converges weakly to an element in the zero set $\operatorname{zer}(A+B)$ as $t \to +\infty$.

\begin{theorem}
\label{thm:weak_convergence_final}
Suppose that Assumptions {\rm {\bf (H1)--(H3)}} and conditions (C1)--(C11) are satisfied. Let $x(t)$ be any strong global solution of system \eqref{eq:SHDS}. Assume in addition that
\[
\text{\rm (C12)}\qquad
\lim_{t\to+\infty}\big(a^2(t)+c(t)\big)=\ell>0,\quad
\dot{\beta}(t)-b(t)=\mathcal{O}\left(\frac{\lambda(t)\dot m(t)}{m^2(t)}\right).
\]
Then, the trajectory $x(t)$ exhibits the following asymptotic properties as $t \to +\infty$:
\begin{enumerate}[label=(\roman*)]
    \item For every $\bar{x} \in \operatorname{zer}(A+B)$, the limit $\lim_{t \to +\infty} \|x(t) - \bar{x}\|_{\mathcal{H}}$ exists;
    \item The trajectory $x(t)$ converges weakly to an element of $\operatorname{zer}(A+B)$.
\end{enumerate}
\end{theorem}

\begin{proof}
Let $\bar{x} \in \operatorname{zer}(A+B)$ be fixed. Since $\mathcal{E}(t)$ is non-increasing and nonnegative by Theorem~\ref{thm:m_weight_bounds}, the limit $\mathcal{E}_{\infty}:=\lim_{t\to+\infty}\mathcal{E}(t)$ exists.

\noindent We first prove that $m(t)v(t)\to0$ as $t \to +\infty$, where $v(t)=\dot{x}(t)+\beta(t)T_{\lambda(t),\gamma(t)}(x(t))$. Set $F(t):=m(t)v(t)$ and $ r(t):=\frac{\dot m(t)}{m(t)}$, then $F(t)$ is locally absolutely continuous.
By condition {\rm (C6)}, $m(t)\to+\infty$ as $t \to +\infty$ and $\int_{t_0}^{+\infty}r(t)dt=+\infty$. From condition {\rm (C10)} and Theorem~\ref{thm:m_weight_bounds}(iv), we have
\[
\int_{t_0}^{+\infty}r(t)\|F(t)\|_{\mathcal{H}}^2dt
=\int_{t_0}^{+\infty}m(t)\dot m(t)\|v(t)\|_{\mathcal{H}}^2dt
\le \frac{1}{C_v}\int_{t_0}^{+\infty}\mathcal{W}_v(t)\|v(t)\|_{\mathcal{H}}^2dt<+\infty.
\]
Moreover, combining Theorem~\ref{thm:m_weight_bounds}(i), condition {\rm (C5)} and Theorem~\ref{thm:m_weight_bounds}(ii) leads to $m(t)v(t)$ is bounded on $[t_0, +\infty)$. 
Substituting the expression of $\dot v(t)$ obtained from system~\eqref{eq:SHDS} into $\dot F(t)=\dot m(t)v(t)+m(t)\dot v(t)$, and using the boundedness of $m(t)v(t)$, Theorem~\ref{thm:pointwise_magnitudes_layered}(i)--(ii), condition {\rm (C8)}, and condition {\rm (C12)}, gives
\begin{align*}
\|\dot F(t)\|_{\mathcal{H}}
&\le \dot m(t)\|v(t)\|_{\mathcal{H}}
 +m(t)|\alpha(t)|\|\dot x(t)\|_{\mathcal{H}}
 +m(t)|\dot\beta(t)-b(t)|\|T_{\lambda(t),\gamma(t)}(x(t))\|_{\mathcal{H}}\\
&=\mathcal{O}\left(\frac{\dot m(t)}{m(t)}\right)
 +m(t)\mathcal{O}\left(\frac{\dot m(t)}{m(t)}\right)\mathcal{O}\left(\frac1{m(t)}\right)
 +m(t)\mathcal{O}\left(\frac{\lambda(t)\dot m(t)}{m^2(t)}\right)\mathcal{O}\left(\frac1{\lambda(t)}\right)\\
&=\mathcal{O}\left(\frac{\dot m(t)}{m(t)}\right)
=\mathcal{O}(r(t)).
\end{align*}
Therefore, applying Lemma~\ref{lem:weighted_decay_criterion} yields $m(t)v(t)\to0$ as $t\to+\infty$.

\noindent We now derive the existence of the distance limit. Expanding the Lyapunov functional gives
\[
\mathcal{E}(t)
=\frac{a^2(t)+c(t)}2\|x(t)-\bar{x}\|_{\mathcal{H}}^2
 +a(t)\langle x(t)-\bar{x},m(t)v(t)\rangle_{\mathcal{H}}
 +\frac12\|m(t)v(t)\|_{\mathcal{H}}^2.
\]
Since $\mathcal{E}(t)\to\mathcal{E}_{\infty}$ and $m(t)v(t)\to0$ as $t\to+\infty$, while $a(t)$ and $x(t)$ remain bounded on $[t_0, +\infty)$, it follows that
\[
\lim_{t\to+\infty}
\frac{a^2(t)+c(t)}2\|x(t)-\bar{x}\|_{\mathcal{H}}^2
=\mathcal{E}_{\infty}.
\]
Combining this with $a^2(t)+c(t)\to \ell>0$ as $t\to+\infty$, we obtain
\[
\lim_{t\to+\infty}\|x(t)-\bar{x}\|_{\mathcal{H}}^2
=\frac{2\mathcal{E}_{\infty}}{\ell}.
\]
Thus $\lim_{t\to+\infty}\|x(t)-\bar{x}\|_{\mathcal{H}}$ exists for every $\bar{x}\in\operatorname{zer}(A+B)$, and the first condition of Opial's Lemma is satisfied.

Now, we turn our attention to the verification of (ii). Let $\tilde{x}$ be a weak sequential cluster point of the trajectory $x(t)$. This implies there exists a sequence $t_n \to +\infty$ such that $x_n := x(t_n) \rightharpoonup \tilde{x}$ weakly as $n \to +\infty$. Define the operator $U_{\gamma} := \operatorname{Id} - J_{\gamma A} \circ (\operatorname{Id} - \gamma B)$. From the result (vi) of Theorem \ref{thm:pointwise_magnitudes_layered}, we have $\|T_{\lambda(t), \gamma(t)}(x(t))\|_{\mathcal{H}} = o(\frac{1}{\lambda(t)})$, which implies that $\|U_{\gamma(t)}(x(t))\|_{\mathcal{H}} = \lambda(t) \|T_{\lambda(t), \gamma(t)}(x(t))\|_{\mathcal{H}} \to 0$ strongly as $t \to +\infty$.

\noindent Since the parameter $\gamma(t)$ is confined within the interval $[\underline{\gamma}, \bar{\gamma}] \subset \left( \max\{0, -2\rho\}, 2\nu \right)$ for all $t \ge t_0$, the sequence $\big( \gamma(t_n) \big)_{n \in \mathbb{N}}$ is bounded. By the Bolzano-Weierstrass theorem, we can extract a convergent subsequence $\big( \gamma(t_{n_k}) \big)_{k \in \mathbb{N}}$ such that
\[
\gamma_{n_k}:=\gamma(t_{n_k})\to\gamma_*\in[\underline{\gamma},\bar{\gamma}],
\qquad x_{n_k}:=x(t_{n_k})\rightharpoonup\tilde{x}
\quad\text{as }k\to+\infty.
\]
Because $U_{\gamma}(x)=\lambda T_{\lambda,\gamma}(x)$ for every $\lambda>0$, Lemma~\ref{lem:sharp_sensitivity_final}(ii), applied with $\lambda_1=\lambda_2=1$, $x=y=x_{n_k}$, $\gamma_1=\gamma_*$ and $\gamma_2=\gamma_{n_k}$, yields
\begin{align}
\|U_{\gamma_*}(x_{n_k})-U_{\gamma_{n_k}}(x_{n_k})\|_{\mathcal{H}}
&\le \frac{2|\gamma_*-\gamma_{n_k}|}{\gamma_*}\|x_{n_k}-\bar{x}\|_{\mathcal{H}}
 +\frac{8\nu|\gamma_*-\gamma_{n_k}|}{\gamma_*}\|B(x_{n_k})\|_{\mathcal{H}}. \label{eq:cluster_final_bound_nk}
\end{align}
As $x_{n_k}\rightharpoonup\tilde{x}$, the sequence $(x_{n_k})$ is bounded; by the Lipschitz continuity of $B$, the sequence $(B(x_{n_k}))$ is also bounded. Together with $\gamma_{n_k}\to\gamma_*$, estimate~\eqref{eq:cluster_final_bound_nk} gives
\[
\|U_{\gamma_*}(x_{n_k})-U_{\gamma_{n_k}}(x_{n_k})\|_{\mathcal{H}}\to0
\qquad \text{as } k\to+\infty.
\]
Moreover, the previously established convergence $\|U_{\gamma(t)}(x(t))\|_{\mathcal{H}}\to0$ as $t\to+\infty$ implies $\|U_{\gamma_{n_k}}(x_{n_k})\|_{\mathcal{H}}\to0$ as $k\to+\infty$. Hence, by the triangle inequality,
\[
U_{\gamma_*}(x_{n_k})\to0
\quad\text{strongly in }\mathcal{H}\quad\text{as }k\to+\infty.
\]
In addition, by Lemma~\ref{thm:coco_T_final} with $\lambda=1$, the operator $U_{\gamma_*}$ is cocoercive; in particular it is monotone and Lipschitz continuous on all of $\mathcal{H}$, hence maximally monotone. Therefore its graph is closed in the weak-strong topology of $\mathcal{H}\times\mathcal{H}$. Since $x_{n_k}\rightharpoonup\tilde{x}$ and $U_{\gamma_*}(x_{n_k})\to0$ as $k\to+\infty$, it follows that $U_{\gamma_*}(\tilde{x})=0$. This is equivalent to $\tilde{x}=J_{\gamma_* A}(\tilde{x}-\gamma_*B\tilde{x})$, and further to $0\in(A+B)(\tilde{x})$. Thus, it follows from Opial's Lemma that the trajectory $x(t)$ converges weakly to an element of $\operatorname{zer}(A+B)$ as $t \to +\infty$.
\end{proof}

\begin{remark}
\label{sec:Analysis of Parameter Conditions}
The convergence conditions {\rm (C1)}--{\rm (C12)} are mutually compatible and can be simultaneously satisfied by the following parameter configuration for $t \ge t_0 \ge 1$ with $g_{\psi}-\theta\Delta/t_0>0$: 
$\phi(t)\equiv k_{\phi}$, $\psi(t)\equiv k_{\psi}$ with $0\le k_{\psi}\le k_{\phi}<1$, 
$g_{\psi}:=1-k_{\psi}$, $\Delta:=k_{\phi}-k_{\psi}\ge 0$, 
$m(t)=t$, $a(t)\equiv a>1$, $c>0$, 
$\gamma(t)\equiv\gamma\in(\max\{0,-2\rho\},2\nu)$, 
$\lambda(t)=\lambda t^{2}$ with $\lambda>0$, 
$\theta(t)=\frac{\theta}{t}$, 
$\delta(t)=\frac{\delta}{t^{q}}$ with $\delta>0$, 
$q:=a+1+\frac{c}{a}>2$, 
$c(t)=a\frac{\frac{\delta}{t^{q-1}}+\frac{\theta\Delta}{t}}{g_{\psi}-\frac{\theta\Delta}{t}}+c$, and
$\theta$ satisfies $0<\theta<\frac{2c\,g_{\psi}\,(a^{2}+c)\,\Omega(\gamma)}{a^{3}}$ where $\Omega(\gamma)=\frac{4\nu(\rho+\gamma)-\gamma^{2}}{4\gamma(\nu+\rho)}$.
\end{remark}

\section{Splitting Algorithm and Numerical Implementation}
\label{sec:algorithm}
This section derives a double inertial Halpern forward-backward splitting algorithm by discretizing system~\eqref{eq:SHDS}. The resulting algorithm is then implemented and evaluated through numerical experiments.

\subsection{From the Continuous System to the Discrete Algorithm}
\paragraph{Time discretization.}
Let
\[
R_{\gamma}(x):=x-J_{\gamma A}(x-\gamma Bx), \qquad
\Phi(t):=\frac{\dot{x}(t)}{\delta(t)}+x(t), \qquad
\Psi(t):=\frac{\Gamma(t)\dot{x}(t)-\theta(t)R_{\gamma(t)}(x(t))}{\delta(t)}.
\]
By the coefficient definitions in~\eqref{eq:coefficients_definition}, system~\eqref{eq:SHDS} is equivalent, for almost every \(t\ge t_0\), to
\[
\frac{d}{dt}\Phi(t)=\frac{d}{dt}\Psi(t).
\]
Thus \(\Phi(t)-\Psi(t)=\hat{x}\) for some \(\hat{x}\in\mathcal H\). Fix a uniform grid \(t_n=t_0+nh\), set \(x_n:=x(t_n)\), and define
\[
\delta_n:=\delta(t_n)h,\quad
\theta_n:=\theta(t_n),\quad
\Gamma_n:=\Gamma(t_n),\quad
\gamma_n:=\gamma(t_n),\quad
\phi_n:=\phi(t_n),\quad
\psi_n:=\psi(t_n).
\]
Introduce
\[
w_n:=x_n+\phi_n(x_n-x_{n-1}),\qquad
y_n:=J_{\gamma_n A}(w_n-\gamma_n B w_n),
\]
so that \(w_n-y_n=R_{\gamma_n}(w_n)\). From now on, the notation \(\mathcal O(h^2)\) refers to the limit \(h\to0\), uniformly for \(t_n\) in each compact interval \([t_0,T]\). Since \(x(t)\) is a strong global solution, \(\dot{x}(t)\) is locally bounded; hence \(x_n-x_{n-1}=\mathcal O(h)\) and \(w_n-x_n=\mathcal O(h)\). The required local boundedness follows from the local boundedness of the trajectory and coefficients, together with the sensitivity estimates for \(R_\gamma\) in Lemma~\ref{lem:sharp_sensitivity_final}.

Applying an endpoint discretization to the identity \(\Phi(t)=\Psi(t)+\hat{x}\) over \([t_{n-1},t_n]\), with the coefficient and residual terms evaluated at \(t_n\) and the zeroth-order term \(x\) evaluated at \(t_{n-1}\), yields the following local consistency relation. The replacement of \(R_{\gamma_n}(x_n)\) by \(R_{\gamma_n}(w_n)\) is absorbed into the \(\mathcal O(h^2)\) term because \(w_n-x_n=\mathcal O(h)\):
\begin{equation}\label{eq:compact_balance}
x_n=\delta_n\hat{x}+(1-\delta_n)x_{n-1}
+\Gamma_n(x_n-x_{n-1})-h\theta_n(w_n-y_n)+\mathcal O(h^2).
\end{equation}
After shifting the index in~\eqref{eq:compact_balance}, the smoothness of the parameter functions and the local sensitivity of \(R_\gamma\) show that replacing the coefficients and residual term at \(t_{n+1}\) by those at \(t_n\) changes the relation only by \(\mathcal O(h^2)\). Hence
\[
x_{n+1}=\delta_n\hat{x}+(1-\delta_n)x_n
+\Gamma_n(x_{n+1}-x_n)-h\theta_n(w_n-y_n)+\mathcal O(h^2).
\]
Moreover, Lemma~\ref{lem:sharp_sensitivity_final}(iii), the local boundedness of the coefficients, and system~\eqref{eq:SHDS} imply that \(\ddot{x}(t)\) is essentially bounded on compact time intervals. Consequently \(\dot{x}(t)\) is locally Lipschitz, and
\[
x_{n+1}-x_n=x_n-x_{n-1}+\mathcal O(h^2).
\]
Since \(\delta_n=\mathcal O(h)\) and the correction \(\Gamma_n(x_n-x_{n-1})-h\theta_n(w_n-y_n)\) is \(\mathcal O(h)\), placing this correction inside the factor \(1-\delta_n\) changes only the \(\mathcal O(h^2)\) remainder. Therefore
\begin{equation}\label{eq:compact_halpern}
x_{n+1}=\delta_n\hat{x}+(1-\delta_n)
\Big[x_n+\Gamma_n(x_n-x_{n-1})-h\theta_n(w_n-y_n)\Big]+\mathcal O(h^2).
\end{equation}

Dropping the local truncation term in~\eqref{eq:compact_halpern}, setting \(h=1\), and identifying \(\hat{x}\) with the anchor point \(x_{\mathrm{anc}}\), we obtain the implementable update. Defining $z_n:=x_n+\psi_n(x_n-x_{n-1})$ and using \(\Gamma_n=\theta_n\phi_n+(1-\theta_n)\psi_n\), the bracket in~\eqref{eq:compact_halpern} becomes
\[
v_n=(1-\theta_n)z_n+\theta_n y_n,
\]
and hence
\[
x_{n+1}=\delta_nx_{\mathrm{anc}}+(1-\delta_n)v_n.
\]
Moreover, for given \(x_{\mathrm{anc}}\) and \(x_0:=x(t_0)\), the initial velocity \(u_0:=\dot{x}(t_0)\) is determined by the identity \(\Phi(t_0)-\Psi(t_0)=x_{\mathrm{anc}}\):
\[
u_0=\frac{\delta(t_0)}{1-\Gamma(t_0)}(x_{\mathrm{anc}}-x_0)
-\frac{\theta(t_0)}{1-\Gamma(t_0)}R_{\gamma(t_0)}(x_0).
\]
With the unit-step approximation \(x_1=x_0+u_0\), this gives
\[
x_1:=x_0+\frac{\delta_0}{1-\Gamma_0}(x_{\mathrm{anc}}-x_0)
-\frac{\theta_0}{1-\Gamma_0}\bigl(x_0-J_{\gamma_0A}(x_0-\gamma_0Bx_0)\bigr).
\]

In addition, the proposed DI-H-FBS algorithm is closely related to several existing splitting schemes through particular parameter choices. If the Halpern anchoring is removed formally, i.e., $\delta_n=0$, then $\phi_n=\psi_n=0$ and $\theta_n=1$ give the classical forward-backward splitting method~\cite{Combettes}, while $\psi_n=0$ and $\theta_n=1$ give the inertial forward-backward splitting scheme~\cite{lorenz2015inertial}. When the Halpern anchoring is kept and the two inertial parameters are identified, i.e., $\phi_n=\psi_n$ and $\theta_n=1$, the update has the form of an inertial Halpern forward-backward splitting method~\cite{cholamjiak2018inertial}.

\begin{algorithm}[H]
\caption{Double Inertial Halpern Forward-Backward Splitting (DI-H-FBS)}
\label{alg:DIH_FBS}
\begin{minipage}{\textwidth}
\begin{tabular}{lp{0.8\linewidth}}
\textbf{Require:} & Anchor point $x_{\mathrm{anc}} \in \mathcal{H}$; starting point $x_0 \in \mathcal{H}$; discrete parameter sequences $\{ \delta_n, \theta_n, \gamma_n, \phi_n, \psi_n \}_{n \ge 0}$ satisfy $\delta_n \in (0, 1), \theta_n \in [0, 1], \phi_n, \psi_n \ge 0$, $\Gamma_n:=\theta_n\phi_n+(1-\theta_n)\psi_n<1$, and $\gamma_n \in (\max\{0, -2\rho\}, 2\nu),$
where $\rho$ and $\nu$ are the comonotonicity and cocoercivity constants of $A$ and $B$, respectively.\\
& Set $x_1 := x_0 + \frac{\delta_0}{1 - \Gamma_0}(x_{\mathrm{anc}} - x_0) - \frac{\theta_0}{1 - \Gamma_0}\bigl(x_0 - J_{\gamma_0 A}(x_0 - \gamma_0 B x_0)\bigr)$.\\
\textbf{Iterate:} & \textbf{for} $n = 1, 2, \dots$ \textbf{do} \\
& \quad $\left| \begin{array}{l}
  \text{--- } w_n = x_n + \phi_n ( x_n - x_{n-1} ) \\
  \text{--- } z_n = x_n + \psi_n ( x_n - x_{n-1} ) \\
  \text{--- } y_n = J_{\gamma_n A} ( w_n - \gamma_n B w_n ) \\
  \text{--- } v_n = ( 1 - \theta_n ) z_n + \theta_n y_n \\
  \text{--- } x_{n+1} = \delta_n x_{\mathrm{anc}} + ( 1 - \delta_n ) v_n
\end{array} \right.$ \\
& \textbf{end for} \\
\textbf{Return:} & The sequence of iterates $\{x_n\}$. \\
\end{tabular}
\end{minipage}
\end{algorithm}

\subsection{Numerical Experiments}\label{sec:Numerical experiments and discrete algorithm}

In this subsection, we conduct several numerical experiments to evaluate the effectiveness and stability of the proposed DI-H-FBS algorithm. The discrete sequences $\{ \delta_n, \theta_n, \gamma_n, \phi_n, \psi_n \}_{n \ge 0}$ are chosen so that the admissibility requirements in Algorithm~\ref{alg:DIH_FBS} are satisfied, namely $\delta_n\in(0,1)$, $\theta_n\in[0,1]$, $\phi_n,\psi_n\ge0$, $\Gamma_n:=\theta_n\phi_n+(1-\theta_n)\psi_n<1$, and the stepsize lies in the corresponding forward-backward range. For DI-H-FBS, the first iterate $x_1$ is generated by the initialization formula in Algorithm~\ref{alg:DIH_FBS}; when a parameter rule is specified only for $n\ge1$, the values of $\delta_0$, $\theta_0$, $\phi_0$, $\psi_0$, and $\gamma_0$ used in the initialization formula are stated separately. The comparison methods are initialized according to their standard implementations, using the same starting point $x_0$ whenever applicable. All computational simulations were implemented using MATLAB R2024a on a computer equipped with an AMD Ryzen 5 4600H processor and 16 GB of RAM. 

\noindent \textbf{Experiment I: Split Feasibility Problem in Function Spaces}

In this experiment, we evaluate the numerical performance of the proposed DI-H-FBS algorithm by solving the Split Feasibility Problem (SFP). The SFP requires finding a point $\overline{x}$ such that:
\begin{equation*}
\overline{x} \in C \quad \text{and} \quad T \overline{x} \in Q,
\end{equation*}
where $C$ and $Q$ are nonempty closed convex subsets of real Hilbert spaces $\mathcal{H}_1$ and $\mathcal{H}_2$, respectively, and $T: \mathcal{H}_1 \to \mathcal{H}_2$ is a bounded linear operator with its adjoint $T^{*}$. This problem is mathematically equivalent to the monotone inclusion problem $0 \in (A+B)x$, where $A = N_C$ is the normal cone operator to $C$, and $B = T^*(I - P_Q)T$ is the gradient of the convex function $f(x) = \frac{1}{2}\|(I - P_Q)Tx\|^2$. It is well-known that $A$ is maximally monotone and $B$ is $\nu$-cocoercive with $\nu = \frac{1}{\|T\|^2}$.

Following Example 5.1 of \cite{cholamjiak2018inertial}, we consider the infinite-dimensional Hilbert space $\mathcal{H}_1=\mathcal{H}_2= L_2[0,1]$ equipped with the inner product $\langle x, y \rangle = \int_0^1 x(t)y(t)dt$. Let $T: \mathcal{H}_1 \to \mathcal{H}_2$ be defined by $(Tx)(t) =\frac{x(t)}{2}$, which implies $\nu = 4$. The constraint sets are defined as $C = \{x(t) \in L_2[0,1] : \langle x, 3t^2 \rangle = 0\}$ and $Q = \{x(t) \in L_2[0,1] : \langle x, \frac{t}{3} \rangle \ge -1\}$. We use two pairs of anchor and starting points: $x_{\mathrm{anc}}(t) = t + 1$, $x_0(t) = 3t^2 - t$, and $x_{\mathrm{anc}}(t) = 1$, $x_0(t) = e^{-t}$. In both cases, the first iterate $x_1$ is generated by Algorithm~\ref{alg:DIH_FBS}, and we set $N = 100$ iterations. The algorithmic performance is evaluated via the residual error $E_n = \frac{1}{2}\|x_n - P_C x_n\|_{L_2}^2 + \frac{1}{2}\|Tx_n - P_Q(Tx_n)\|_{L_2}^2$ using the specific parameters $\phi_n = 0.9$, $\psi_n = 0.4$, $\theta_n = 0.42/n$, $\gamma_n = 1.0$, and $\delta_n = \min(0.5, 0.1/n^{3.93})$. The values $\delta_0=0.1$, $\theta_0=0.42$, $\phi_0=0.9$, $\psi_0=0.4$, and $\gamma_0=1.0$ are used in the initialization formula.

Furthermore, we compare the DI-H-FBS algorithm against the I-H-FBS algorithm~\cite{cholamjiak2018inertial}, which is recovered from DI-H-FBS  by setting $\theta_n=1$ and $\phi_n=\psi_n=0.9$, and also against the non-inertia DI-H-FBS algorithm (N-I-H-FBS) as benchmarks. As illustrated in Figure~\ref{fig:SFP_convergence} with the logarithmic residual plot, N-I-H-FBS (yellow) exhibits a slow sublinear convergence rate. While I-H-FBS (red) accelerates the process, it generates significant numerical oscillations. In contrast, DI-H-FBS (blue) effectively suppresses these fluctuations through its decoupled Hessian-driven damping branch ($\psi_n = 0.4$), achieving much smoother, faster convergence toward the solution. 

\begin{figure}[h!]
\centering
\includegraphics[width=0.8\textwidth]{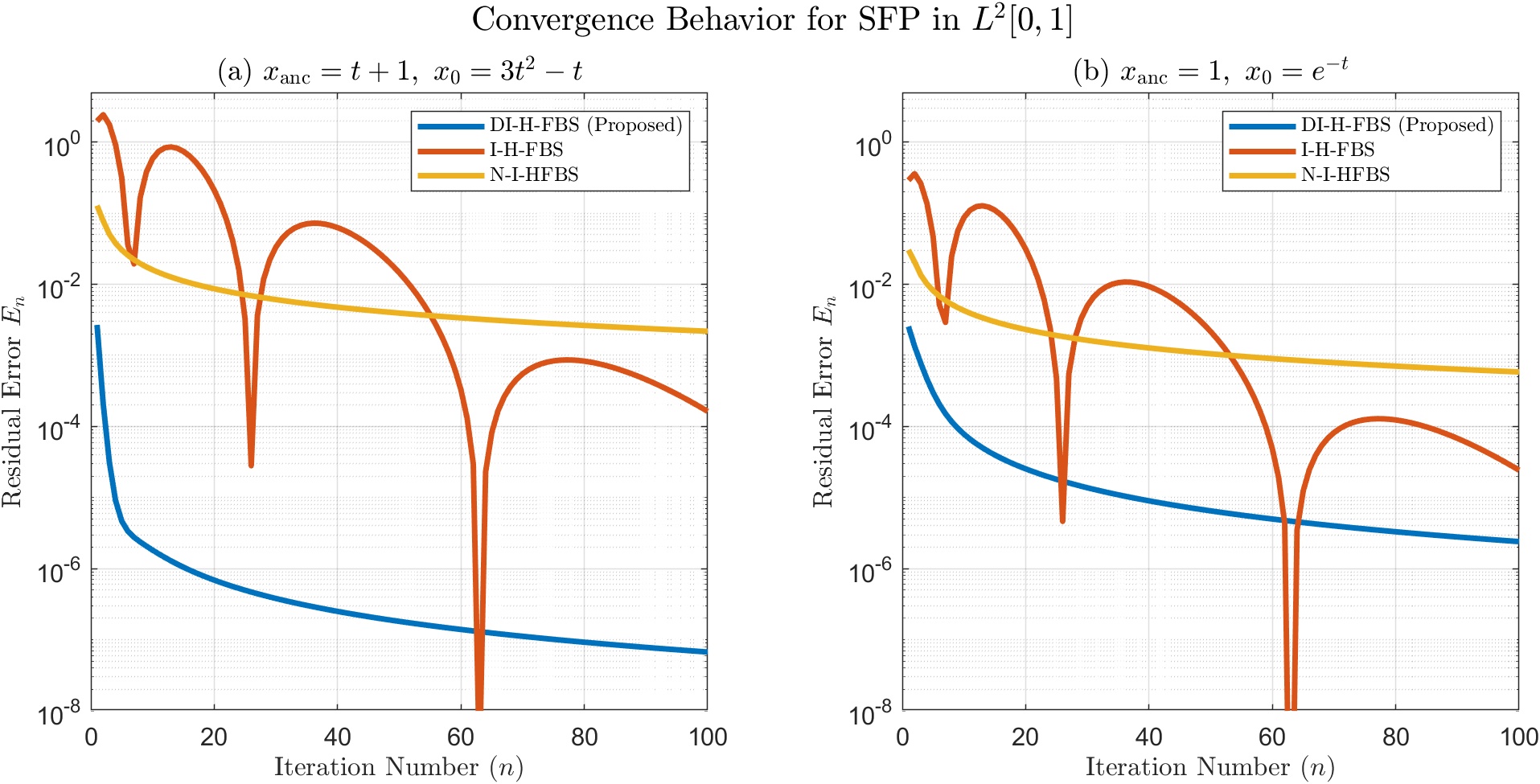}
\caption{Evolution of the logarithmic residual error $E_n$ for the Split Feasibility Problem in $L_2[0,1]$.}
\label{fig:SFP_convergence}
\end{figure}

\noindent \textbf{Experiment II: Regularized Inverse Problems in Signal and Image Processing}

Consider the general regularized inverse problem formulated as:
\begin{equation}\label{eq:composite_opt}
\min_{x \in \mathbb{R}^N} \left\{ \frac{1}{2}\| \mathcal{A}x - b \|_2^2 + \mu \|x\|_1 \right\},
\end{equation}
where $\mathcal{A}: \mathbb{R}^N \to \mathbb{R}^M$ is a linear sensing or blurring operator, $b \in \mathbb{R}^M$ represents the degraded observations corrupted by additive noise, and $\mu > 0$ is regularization parameter controlling the sparsity of the solution. 

By Fermat's rule, problem \eqref{eq:composite_opt} is equivalent to finding a zero of the sum of two operators, i.e., $0 \in (A+B)x^*$, where:
\begin{equation*}
B(x) := \nabla \left( \frac{1}{2}\| \mathcal{A}x - b \|_2^2 \right) = \mathcal{A}^T(\mathcal{A}x - b), \quad \text{and} \quad A(x) := \partial (\mu \|x\|_1).
\end{equation*}
It is well-established that the subdifferential operator $A$ is maximally monotone. Furthermore, since the data-fidelity term is a convex quadratic function, its gradient $B$ is Lipschitz continuous with the constant $L = \|\mathcal{A}^T \mathcal{A}\|$. According to the Baillon-Haddad theorem, $B$ is inherently $\nu$-cocoercive with $\nu = \frac{1}{L}$.

To overall benchmark our method against state-of-the-art inertial splitting algorithms---namely, FISTA \cite{beck2009fast}, FISTA-CD \cite{chambolle2015convergence}, cGIGPM \cite{wu2019general}, and DIPFBSM \cite{jolaoso2023double}---we present two specific application scenarios. In these experiments, the rule $\phi_n=(n-1)/(n+2)$ used in DI-H-FBS follows the standard Chambolle--Dossal/FISTA-CD inertial schedule.

\noindent \textbf{Example 1: LASSO Signal Reconstruction}

In this example, we apply the algorithms to the Least Absolute Shrinkage and Selection Operator (LASSO) problem within the context of compressed sensing. The goal is to accurately reconstruct a sparse signal from under-determined measurements ($M \ll N$). 

For the benchmark protocol, we set the dimensions to $N = 4096$ and $M = 1024$. The sensing matrix $\mathcal{A}$ is generated with entries drawn from a normal distribution $\mathcal{N}(0, \frac{1}{M})$ to ensure the Restricted Isometry Property (RIP). The ground-truth sparse signal $x^*$ is synthetically constructed to contain $160$ non-zero entries uniformly distributed over $[-2, 2]$. The observation vector $b$ is contaminated with Gaussian noise, calibrated to yield a Signal-to-Noise Ratio (SNR) of $40$ dB. The regularization parameter is set to $\mu = 0.02$. 

Furthermore, DI-H-FBS employs the discrete parameters \(\phi_n = (n-1)/(n+2)\) for \(n\ge1\), \(\psi_n = 0.5\), \(\theta_n = 0.995\), \(\delta_n = 10^{-4}/n^{13}\), and \(\gamma_n = 0.35/L\). For the initialization formula, we take \(\delta_0=10^{-4}\), \(\theta_0=0.995\), \(\phi_0=0\), \(\psi_0=0.5\), and \(\gamma_0=0.35/L\). FISTA is configured with \(\gamma = 0.5/L\). FISTA-CD uses \(\gamma = 0.5/L\) and \(\beta_n = (n-1)/(n+2)\) (i.e., \(a=3\)). For cGIGPM, we set \(\alpha = 0.35\), \(\beta = 0.50\), \(\lambda_0 = 0.03\), \(\varepsilon = 10^{-6}\), and \(\gamma_n = \min\{\lambda_0,\, m/L\}\) with \(m = \min\{\alpha/\beta,\,(2-2\alpha-\varepsilon)/(1-\beta)\}\). DIPFBSM is implemented with \(\gamma = 1.0/L\), \(\theta = 0.49\), and \(\delta = -0.5\). The anchor point is fixed as \(x_{\mathrm{anc}} = \mathbf{0}\), and \(x_0=\mathbf{0}\) is used as the starting point; for DI-H-FBS, \(x_1\) is generated by Algorithm~\ref{alg:DIH_FBS}. Let $x_n$ denote the $n$-th iterate generated by each algorithm. The performance is evaluated using the Mean Squared Error (MSE) metric, defined as $\mathrm{MSE} = \frac{1}{N}\|x_n - x^*\|_2^2$.

As shown in Figure~\ref{fig:LASSO_convergence}, DI-H-FBS achieves the lowest MSE (\(3.25\times10^{-5}\)) among all methods, with a CPU time comparable to the fastest competitors. FISTA and FISTA-CD yield close but slightly worse accuracy, while cGIGPM and DIPFBSM exhibit significantly larger errors, highlighting the efficiency and stability of the proposed algorithm.

\begin{figure}[h!]
\centering
\includegraphics[width=\textwidth]{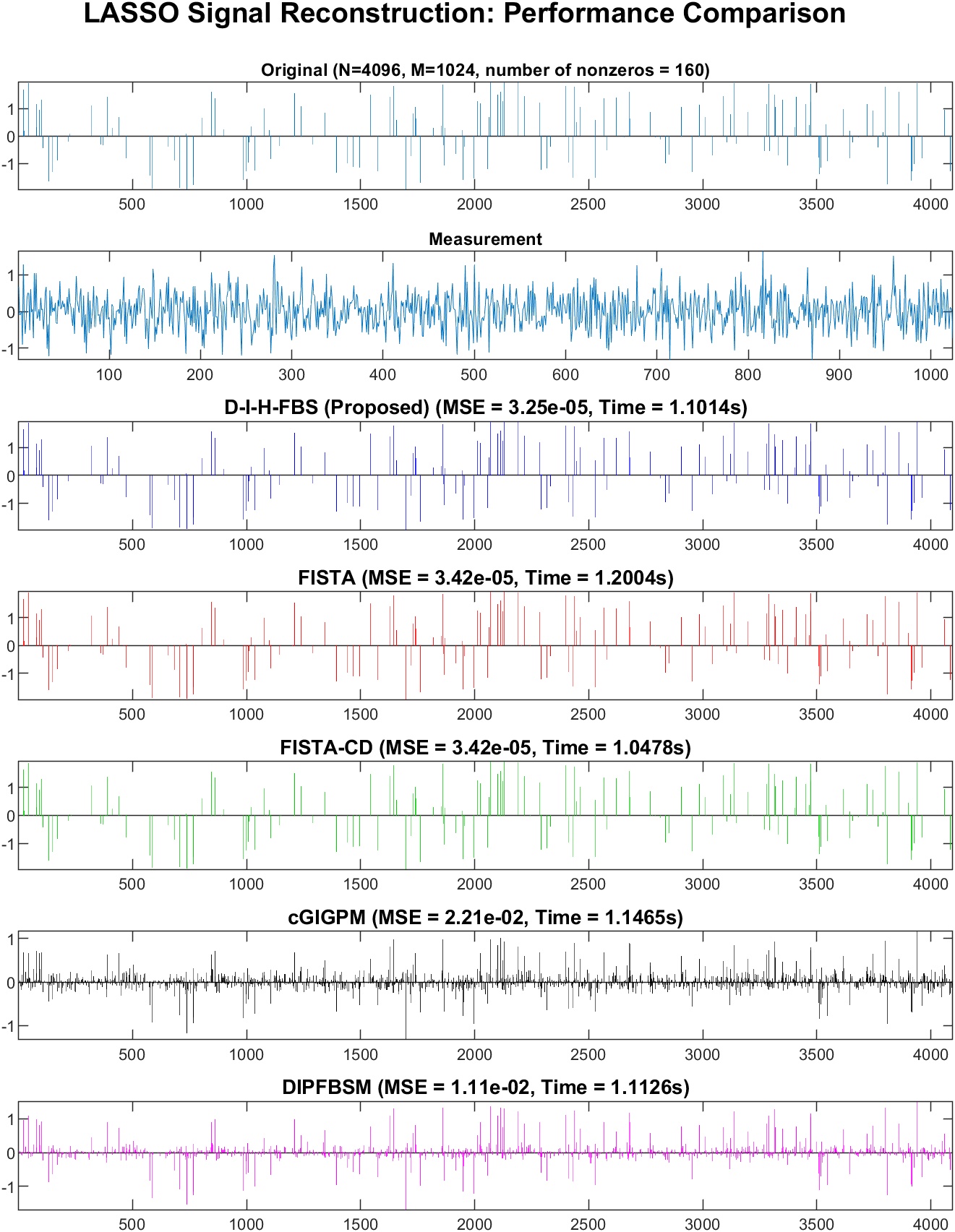}
\caption{Convergence of the mean squared error (MSE) for the LASSO signal reconstruction problem.}
\label{fig:LASSO_convergence}
\end{figure}

\noindent \textbf{Example 2: Image Deblurring}

In the second scenario, we consider the classical image deblurring problem. We take the standard test image (a $512\times512$ scenery image) and convert it to grayscale, normalizing pixel values to $[0,1]$. The degradation model consists of a Gaussian blur kernel of size $7\times7$ with standard deviation $\sigma=4$, followed by additive Gaussian noise with $\mathrm{SNR}=40$ dB. The linear operator $\mathcal{A}$ corresponds to the 2D convolution with the blur kernel under circular boundary conditions.

Furthermore, the regularization parameter is fixed to $\mu = 8\times10^{-4}$. For DI-H-FBS, we set $\phi_n = (n-1)/(n+2)$ for \(n\ge1\), $\psi_n = 0.3$, $\theta_n = 0.98$, $\delta_n = 0.01/n^{5.2}$, and $\gamma_n = 0.05/L$. For the initialization formula, we take \(\delta_0=0.01\), \(\theta_0=0.98\), \(\phi_0=0\), \(\psi_0=0.3\), and \(\gamma_0=0.05/L\). FISTA-CD employs $\gamma = 0.1/L$ and the same inertial rule as above. The cGIGPM parameters are identical to those in Example~1. DIPFBSM uses $\gamma = 1.31/L$, $\theta = 0.6$, $\delta = -0.5$. All algorithms are initialized with the blurred and noisy observation $b$, and the anchor point is set to $b$ for the proposed method; for DI-H-FBS, the first iterate is generated by Algorithm~\ref{alg:DIH_FBS}. We assess the quality of the restored images by means of the Peak Signal-to-Noise Ratio (PSNR) and the Structural Similarity Index (SSIM), defined respectively as
\[
\mathrm{PSNR}(x_r, x) = 10 \log_{10} \left( \frac{1}{\frac{1}{N} \| x_r - x \|_2^2 } \right), \qquad
\mathrm{SSIM}(x_r, x) = \frac{ (2\mu_x \mu_{x_r} + C_1)(2\sigma_{xx_r} + C_2) }{ (\mu_x^2 + \mu_{x_r}^2 + C_1)(\sigma_x^2 + \sigma_{x_r}^2 + C_2) },
\]
where $N$ is the number of pixels, $\mu_x$ and $\sigma_x^2$ denote the mean and variance of the original image $x$, $\sigma_{xx_r}$ is the covariance between $x$ and the restored image $x_r$, and the constants are $C_1 = (0.01L)^2$, $C_2 = (0.03L)^2$ with $L = 1$ for images normalized to $[0,1]$.

As shown in Figures~\ref{fig:deblur_metrics} and \ref{fig:deblur_visual}, DI-H-FBS achieves the best final restoration quality among all compared methods, with PSNR \(27.96\) dB and SSIM \(0.750\). The metric curves also show that the proposed method remains stable and competitive throughout the iteration process.

\begin{figure}[h!]
\centering
\includegraphics[width=0.8\textwidth]{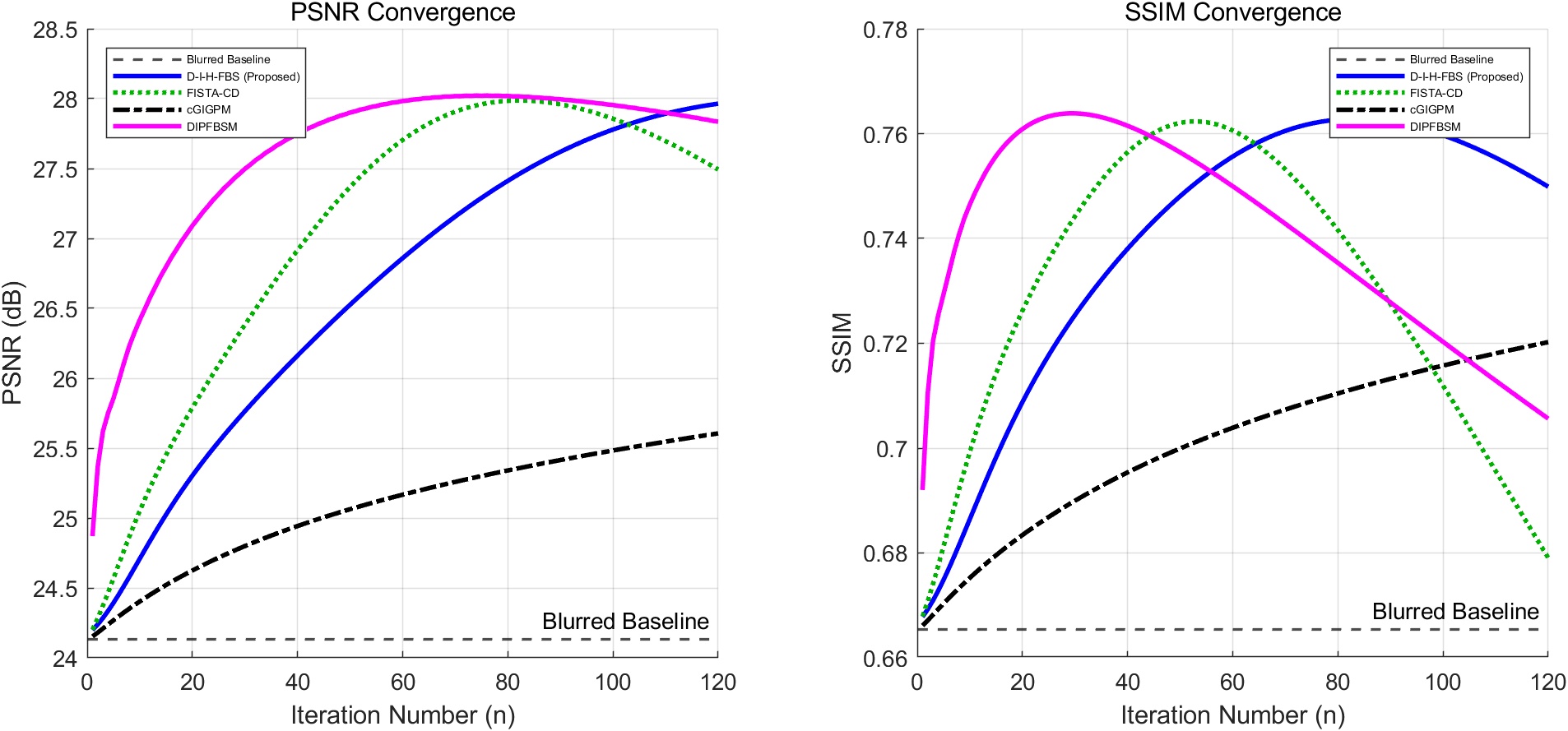}
\caption{Convergence of PSNR (left) and SSIM (right) for the image deblurring problem.}
\label{fig:deblur_metrics}
\end{figure}

\begin{figure}[h!]
\centering
\includegraphics[width=\textwidth]{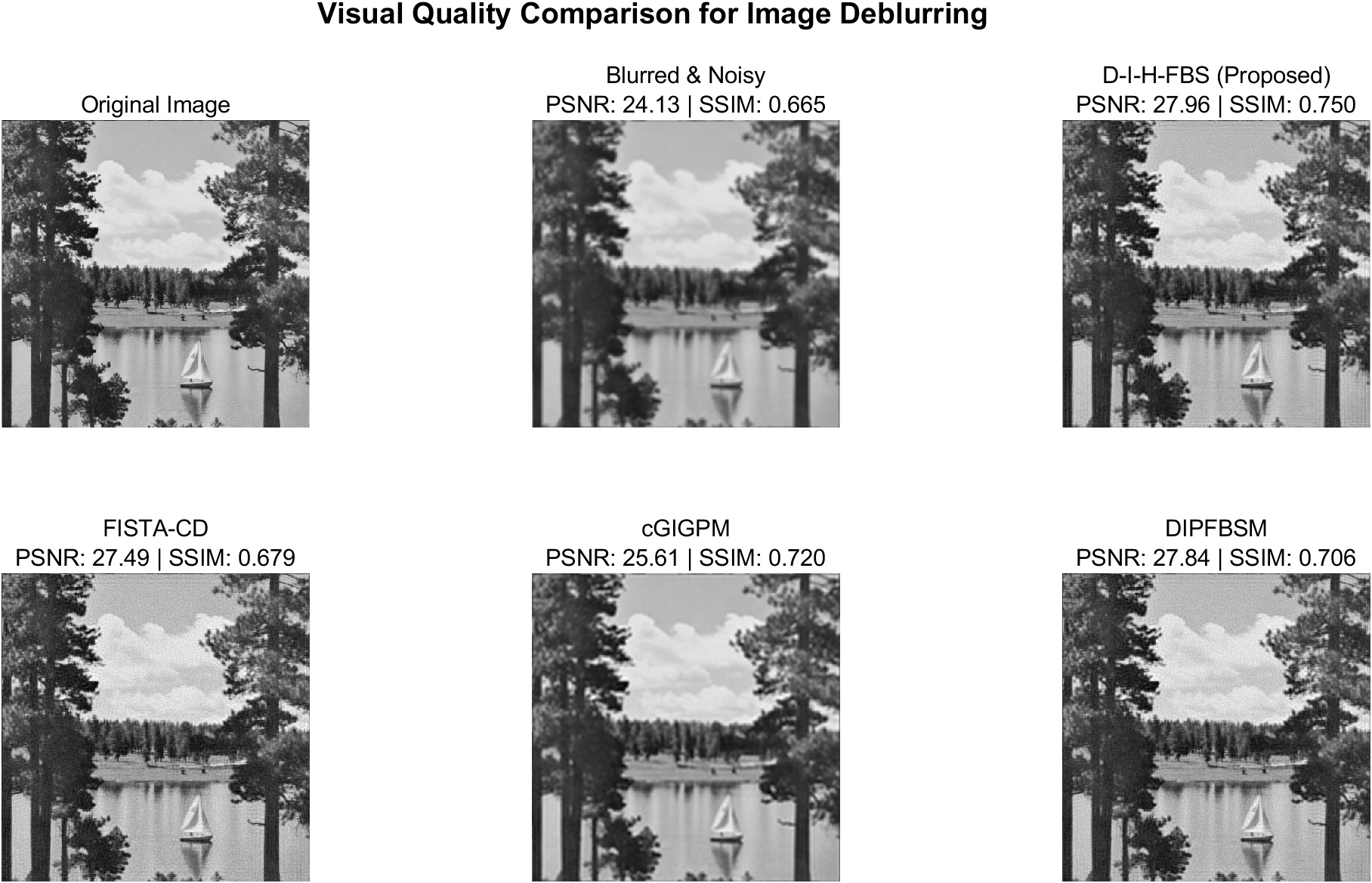}
\caption{Visual quality comparison of the restored images. From left to right and top to bottom: original, blurred \& noisy, DI‑H‑FBS (proposed), FISTA‑CD, cGIGPM, DIPFBSM.}
\label{fig:deblur_visual}
\end{figure}

\section{Conclusion}
\label{sec:conclusion}

This paper has extended the second-order dynamical systems approach to the setting of maximal $\rho$-comonotone operators, proving that the velocity and the forward-backward residual converge at the accelerated rates $o(1/t)$ and $o(1/t^{2})$, respectively, and that the trajectories converge weakly to $\operatorname{zer}(A+B)$. The corresponding discrete algorithm, the Double Inertial Halpern Forward-Backward Splitting method, was derived by temporal discretization and was shown numerically to perform favorably on split feasibility, sparse recovery, and deblurring problems. Future work includes the study of strong convergence and extensions to cases where both operators are governed by generalized monotonicity assumptions.

\section*{Acknowledgments}

The authors express gratitude to the anonymous reviewers and the editor for their valuable feedback and suggestions, which significantly enhanced the earlier draft of this manuscript. This work was supported by the Team Building Project for Graduate Tutors in Chongqing (Grant No. yds223010), the Chongqing Natural Science Foundation (Grant No. CSTB2025NSCQ-GPX0814), and the Scientific Research Start-up Fund Project of Chongqing Technology and Business University (Grant No. 2656003).

\section*{CRediT Authorship Contribution Statement}

\textbf{Yan Tang}: Conceptualization, Methodology, Formal analysis, Supervision, Project administration, Funding acquisition. \textbf{Jun Dong}: Software, Validation, Visualization, Investigation, Data curation, Writing -- original draft.

\section*{Declaration of Competing Interest}

The authors declare that they have no known competing financial interests or personal relationships that could have appeared to influence the work reported in this paper.

\end{document}